\documentclass[10pt]{article}
\usepackage[a4paper,left=25mm,right=25mm,top=25mm,bottom=25mm]{geometry}

\usepackage[utf8]{inputenc}

\RequirePackage[numbers]{natbib}

\usepackage{amsthm,amsmath,amsfonts,amssymb}
\usepackage{graphicx,psfrag,epsf}
\usepackage{enumerate}
\usepackage{url}
\usepackage{hyperref}
\hypersetup{colorlinks=true,allcolors=[rgb]{0,0,1}}
\usepackage{color,xcolor}
\usepackage{dsfont}
\usepackage{multirow}
\usepackage{array}
\usepackage{booktabs}
\usepackage{caption}
\usepackage{subcaption}
\allowdisplaybreaks
\selectfont

\theoremstyle{plain}
\newtheorem{theorem}{Theorem}[section]
\newtheorem{proposition}[theorem]{Proposition}
\newtheorem{cor}[theorem]{Corollary}
\newtheorem{lem}[theorem]{Lemma}

\theoremstyle{definition}

\newtheorem{example}[theorem]{Example}

\theoremstyle{remark}
\newtheorem{remark}[theorem]{Remark}

\newcommand{\R}{\mathbb{R}}
\newcommand{\PP}{\mathbb{P}}

\newcommand{\XX}{\mathbf{X}}

\newcommand{\xx}{\mathbf{x}}

\newcommand{\Var}{\mathrm{Var}}

\usepackage{bbm}

\usepackage{tikz}
\usepackage{pgfplots}

\usepackage{wrapfig}
\usetikzlibrary{positioning,arrows}
\usetikzlibrary{arrows.meta,decorations.markings}
\usetikzlibrary{matrix}
\usepackage{hologo}

\usepackage{xcolor}
\usetikzlibrary{matrix,fit}
\tikzset{
  mleftdelimiter/.style={
    inner ysep=0pt,
    inner xsep=1ex,
    left delimiter=\{,
    label={[label distance=3mm]left:#1}
  }
}

\definecolor{light-gray}{gray}{0.95}
\definecolor{darkblue}{rgb}{0,0,.5}
\definecolor{foxred}{rgb}{0.7, 0.11, 0.11}

\begin{document}

\title{Dependence functions based on Chatterjee's rank correlation}

\author{
Carsten Limbach\\
Paris Lodron Universität Salzburg\\
Department for Artificial Intelligence and Human Interface\\
Hellbrunner Straße 34, 5020 Salzburg, Austria\\
\texttt{carsten.limbach@plus.ac.at}
}

\date{13.05.2026}

\maketitle

\begin{abstract}
{We investigate a geometric and distributional reinterpretation of Chatterjee’s
$\xi$-coefficient, a measure of functional dependence between a response variable
$Y$ and a predictor vector $\XX$. To this end, we analyze the Markov
product $(Y, Y')$, where $Y'$ denotes a copy of $Y$ that is
conditionally independent of $Y$ given $\XX$, and we define and study two dependence functions, $\phi_{(Y,\XX)}$ and $\kappa_{(Y,\XX)}$.
This approach provides geometric intuition for the structure of the Markov
product and extends Chatterjee’s correlation coefficient to a richer and more interpretable object for studying directed stochastic dependence.
In particular, beyond asking how well \(Y\) can be described as a function of \(\XX\),
the proposed functions allow one to assess the concentration of the Markov product
near the diagonal.
}
\end{abstract}

\medskip

\section{Introduction}

Dependence measures play a central role in quantifying different types of relationships between random variables. While Pearson’s correlation measures linear dependence and measures of concordance such as Kendall’s $\tau$ and Spearman’s $\rho$ capture monotone dependence, measures such as Chatterjee’s correlation coefficient $\xi$ are designed to detect functional dependence.

Chatterjee’s dependence coefficient has stimulated a growing body of work on its theoretical properties, extensions, and variants. Recent contributions study, for instance, continuity and monotonicity properties of $\xi$ \cite{ansari2025continuity,ansari2026ordering,deb2020chatterjee,tran2024rank}, properties of its estimator \cite{dalitz2024bias,lin2023boosting,shi2021azadkia}, and extensions to multivariate or multi-response settings \cite{ansari2022extension,huang2023multivariate}.

Chatterjee’s dependence coefficient admits the representation introduced in \cite{azadkia2021simple}:
\begin{align}
\xi(Y,\mathbf{X})
:=
\frac{\int_{\mathbb{R}} \Var\bigl(\PP(Y \ge y \mid \mathbf{X})\bigr)\,
\mathrm{d}\PP^{Y}(y)}
{\int_{\mathbb{R}} \Var\bigl(\mathbf{1}_{\{Y \ge y\}}\bigr)\,
\mathrm{d}\PP^{Y}(y)}.
\end{align}
The coefficient takes values in $[0,1]$ \cite{azadkia2021simple}; it vanishes if and only if $Y$ and $\XX$ are independent, and it equals one if and only if $Y$ perfectly depends  on $\XX$, that is, if $Y=f(\XX)$ almost surely for some measurable function $f$.

Throughout the paper, \(F_Y\) is assumed to be continuous. Moreover, \((Y,Y')\) denotes the conditional i.i.d. Markov product given \(\XX\); that is, conditionally on \(\XX\), the random variables \(Y\) and \(Y'\) are independent and identically distributed. Chatterjee’s rank correlation then admits the following representation; see \cite{ansari2025continuity} and the calculation in Theorem~\ref{Intro}:
\begin{equation}
\xi(Y,\XX)
=
6 \int_{\mathbb{R}} \mathbb{P}(Y \ge y, Y' \ge y)\,\mathrm{d}\PP^{Y}(y) - 2.
\end{equation}
It is the basis for the nearest-neighbor rank estimator via the Markov-product approximation of $\xi(Y,\XX)$.

The usefulness of this reformulation is reflected in the fact that $\xi(Y,\XX)=1$ holds if and only if $Y$ and $Y'$ are comonotone, or, equivalently, if $Y=Y'$ almost surely, while $\xi(Y,\XX)=0$ is equivalent to the independence of $Y$ and $Y'$, which in turn is equivalent to the independence of $\XX$ and $Y$, see \cite{fuchs2026dimension}.

\begin{theorem}\label{Intro}
Assume that the distribution function $F_Y$ is continuous.  
Then the following representations of $\xi(Y,\XX)$ are equivalent:
\begin{equation}\label{intor2s}
\xi(Y,\XX)
=
6 \int_{\mathbb{R}} \mathbb{P}(Y \ge y, Y' \ge y)\,\mathrm{d}\PP^{Y}(y) - 2
= 3\int_0^1 \PP\bigl(|F_Y(Y)-F_Y(Y')|\le s\bigr)\,ds - 2.
\end{equation}
\end{theorem}

The representation in Theorem~\ref{Intro} expresses $\xi$ in
terms of probability integral transforms and measures dependence through
the expected absolute difference between transformed observations. This representation shows that $\xi$ depends on the distribution of
\(
\lvert F_Y(Y) - F_Y(Y') \rvert
\)
under the Markov product.
It can be expressed in terms of the following two functions:
\begin{align}
\phi_{(Y,\mathbf{X})}(t)
&:= \PP\bigl(|F_Y(Y)-F_Y(Y')|\le t\bigr), \label{Phi1}\\
\kappa_{(Y,\mathbf{X})}(t)
&:= 1 - 3\int_0^t (1-\phi_{(Y,\mathbf{X})}(s))\,\mathrm{d}s. \label{Phi2}
\end{align}
We note that $\phi_{(Y,\XX)}(1)=1$, $\kappa_{(Y,\XX)}(0)=1$, and $\kappa_{(Y,\XX)}(1)=\xi(Y,\XX)$.
\\

The function $\phi_{(Y,\mathbf{X})}$ corresponds to the integrand appearing in the representation of Chatterjee’s correlation coefficient in (\ref{intor2s}). In contrast, $\kappa_{(Y,\mathbf{X})}$ arises when the integration is restricted to the interval $[0,t]$ instead of the full range $[0,1]$. This localized version captures additional structural information about the distribution of $|F_Y(Y)-F_Y(Y')|$ and thus provides a more refined view on the dependence between $Y$ and $\mathbf{X}$.

In this work, we investigate the limitations of Chatterjee’s rank correlation in terms of its restricted explanatory power and explain why $\phi_{(Y,\XX)}(0)$ constitutes a helpful quantity for understanding dependence structures. In particular, Chatterjee’s correlation coefficient is intended to quantify how well dependence can be characterized in a functional manner. However, the examples shown in Figure~\ref{fig:drei_bilder23} raise the question of how effectively this measure captures such structure: all three cases attain the same value $\xi(Y,\XX)=0.25$, despite exhibiting substantial differences in their underlying functional forms. Instead, Figure \ref{fig:zwei_bilder5} depicts the functions $\phi_{(Y,X)}$ and $\kappa$ for each of the dependence structures shown in Figure \ref{fig:drei_bilder23}, thereby revealing significant differences.

\begin{figure}[ht]
\centering
\includegraphics[width=0.85\textwidth]{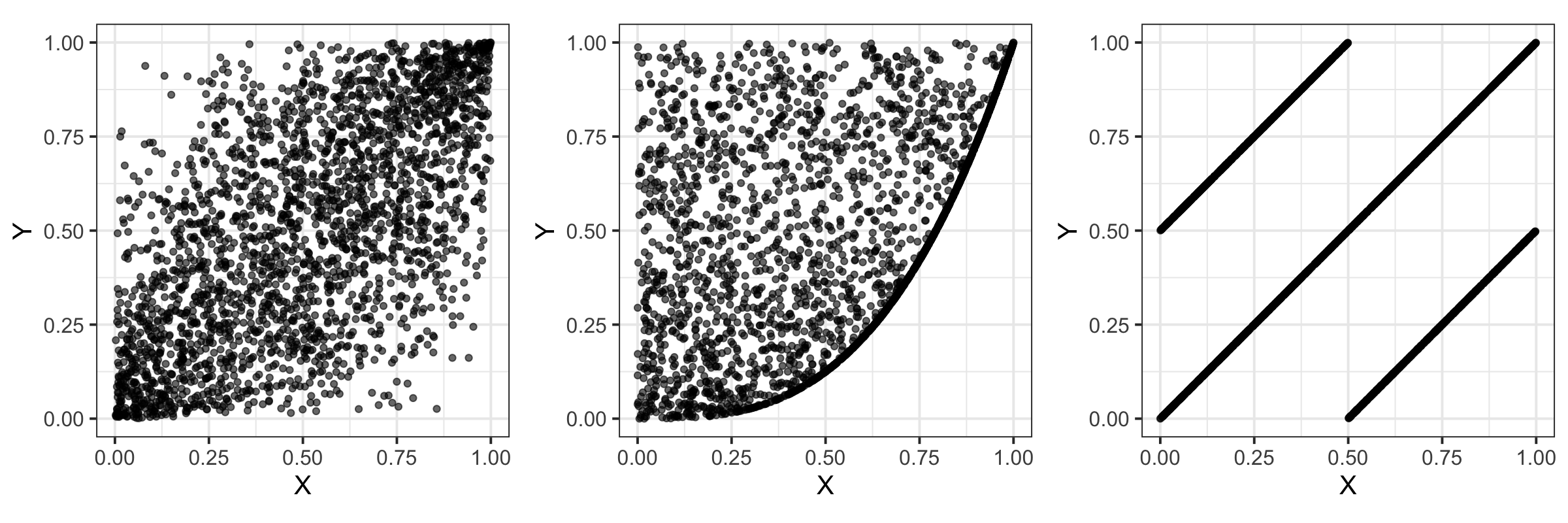}
\caption{From left to right: the random vector $(Y,X)$ associated with a Gaussian copula (\ref{F_Yaus}) with parameter $\sqrt{\sqrt{2}-1}$, a Marshall--Olkin copula (\ref{Olkin}) with parameters $(1,\,1/3)$, and a Jump copula with parameter $1$ (see Example~\ref{Jump}), based on a sample size of $n = 2500$.}
\label{fig:drei_bilder23}
\end{figure}

\begin{figure}[ht]
\centering
\includegraphics[width=0.67\textwidth]{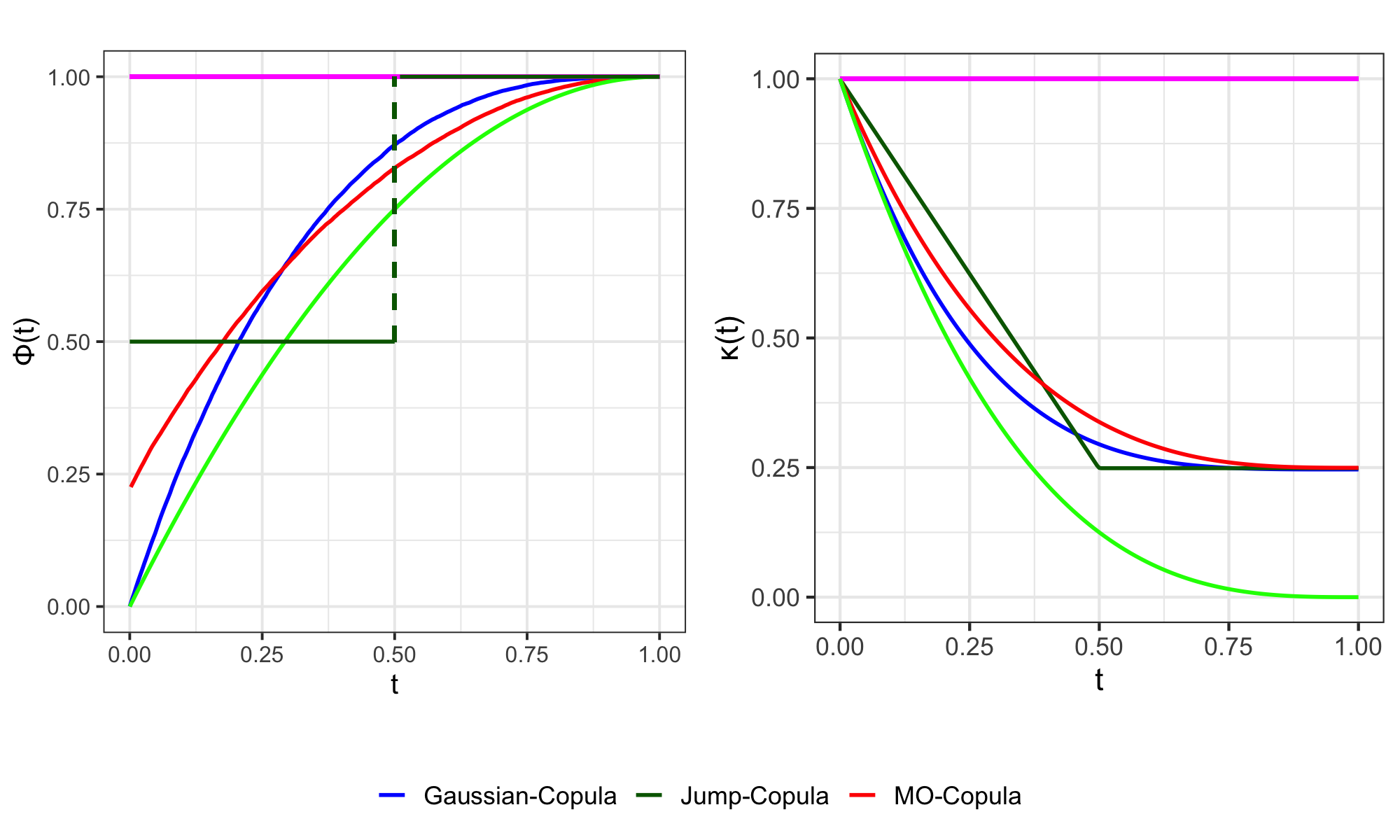}
\caption{The dependence functions $\phi_{(Y,X)}$ and $\kappa$ for a random vector $(Y,X)$ associated with a Gaussian copula (\ref{F_Yaus}) with parameter $\rho=\sqrt{\sqrt{2}-1}$ (blue), a Marshall--Olkin copula (\ref{Olkin}) with parameters $(\alpha,\beta)=(1,\tfrac{1}{3})$ (red), and a Jump copula (\ref{Jump}) with parameter $1$ (green).}
\label{fig:zwei_bilder5}
\end{figure}

We have the following characterizations of independence and perfect dependence: $\phi_{(Y,\XX)}(t)=2t-t^2 \text{ for all } t\in[0,1] \text{ if and only if } \XX \perp Y,
$ whereas $\phi_{(Y,\XX)}(t)=1 \text{ for all } t\in[0,1] \text{ if and only if } Y=f(\XX) \text{ almost surely.}$
Similarly, by Theorem~\ref{ter2}, there exists \( t \in (0,1] \) such that
\(\kappa_{(Y,\XX)}(t)=0\) if and only if \( \XX \perp Y \), while $\kappa_{(Y,\XX)}(t)=1 \text{ for all } t\in(0,1]$ if and only if \( Y \) depends perfectly on \( \XX \).
We denote by $\phi_{(Y,\XX)}^{\perp}$ and $\kappa_{(Y,\XX)}^{\perp}$ the functions corresponding to independence of $\XX$ and $Y$, and by $\phi_{(Y,\XX)}^{pd}$ and $\kappa_{(Y,\XX)}^{pd}$ those corresponding to perfect dependence of $Y$ on $\XX$, that is,
\[
\phi_{(Y,\XX)}^{\perp}(t) = 2t - t^2, \qquad
\kappa_{(Y,\XX)}^{\perp}(t) = (1-t)^3 \ \text{ for all } t \in [0,1],
\]
and
\[
\phi_{(Y,\XX)}^{pd}(t) = 1, \qquad
\kappa_{(Y,\XX)}^{pd}(t) = 1 \ \text{ for all } t \in [0,1].
\]

Of particular interest is the interpretation of $\phi_{(Y,\XX)}$ at the point $0$. 
Note that this corresponds to the event $V=V'$, where $V=F_Y(Y)$ and $V'=F_Y(Y')$. 
While this includes the case $Y=Y'$, the quantity $\phi_{(Y,\XX)}(0)$ should not be 
interpreted as the probability mass of a deterministic relation $Y=f(\XX)$.

A more precise identity is
\[
\phi_{(Y,\XX)}(0)
=
\PP(V=V')
=
\mathbb{E}\!\left[
\sum_{a} \PP(V=a\mid \XX)^2
\right],
\]
where the sum runs over all atoms $a$ of the conditional distribution of $V$ given $\XX$.

Thus, $\phi_{(Y,\XX)}(0)$ measures the \emph{squared conditional atomic mass} of $V$ given $\XX$, rather than the probability of a purely functional component. In particular, in mixture models a deterministic component with weight $p$ contributes $p^2$, not $p$.

Consequently, $\phi_{(Y,\XX)}(0)$ quantifies the amount of atomic mass in the conditional law of $F_Y(Y)$ given $\XX$. It detects exact conditional ties in the associated Markov product and is strictly positive whenever the conditional distribution has point-mass components on a set of positive \(\PP^{\mathbf X}\)-measure.

The analysis of $\phi_{(Y,X)}$ also reveals several desirable properties, including the convexity of $\kappa_{(Y,\XX)}$.

The above approach further allows a natural \emph{geometric perspective}. The origin of this viewpoint lies in the following observation: under the Markov product transformation, perfect functional dependence of $Y$ on $\mathbf{X}$ implies $F_Y(Y) = F_Y(Y')$ almost surely, so that all mass concentrates on the diagonal of the unit square. Heuristically, stronger dependence between \(\mathbf X\) and \(Y\) can be interpreted as making the conditionally sampled pair \((Y,Y')\), given \(\mathbf X\), more concentrated near the diagonal. This informal interpretation should not be read as a general monotonicity result; the shuffle examples below illustrate that non-monotone behavior may occur. For $t \in [0,1]$, define the set
\begin{align}
A_t 
:= \{(u,v) \in [0,1]^2 : |u - v| \le t\},
\end{align}
see Figure~\ref{A} for an illustration. Then
\begin{align}
\phi_{(Y,\mathbf{X})}(t)
= \mathbb{P}^{(F_Y(Y),\,F_Y(Y'))}(A_t). \label{T12}
\end{align}
In other words, $\phi_{(Y,\mathbf{X})}(t)$ measures the probability mass of the transformed pair $(F_Y(Y), F_Y(Y'))$ contained in a band of width $2t$ around the diagonal. In this geometric interpretation, certain forms of stronger dependence may manifest themselves through a higher concentration of mass near the diagonal, which can be reflected in larger values of \(\phi_{(Y,\mathbf X)}(t)\) for small \(t\).

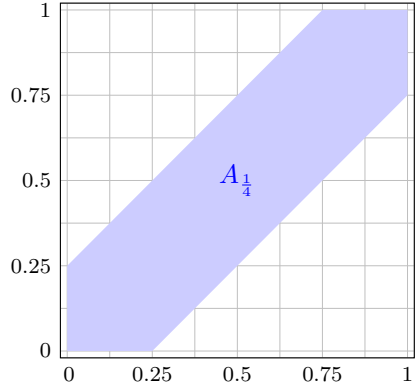
\begin{figure}[h]
    \centering
    \begin{tikzpicture}

        \begin{scope}[yshift=2.5cm,xshift=2.5cm]

            \begin{scope}[shift={(0,-4.3)}, scale=4.5]

                \foreach \x in {-0.02,0,0.125,...,1,1.02} {
                    \draw[gray!50, very thin] (\x,-0.02) -- (\x,1.02);  
                }
                \foreach \y in {-0.02,0,0.125,...,1,1.02} {
                    \draw[gray!50, very thin] (-0.02,\y) -- (1.02,\y);  
                }

                \fill[blue!20] (0,0.25) -- (0.75,1) -- (1,1) -- (1,0.75) -- (0.25,0) -- (0,0);

                \node[text=blue, font=\bfseries] at (0.5,0.5) {$A_{\frac{1}{4}}$};

                \draw (-0.02,-0.02) rectangle (1.02,1.02);

                \node[below left] at (0.05,-0.02) {\footnotesize $0$};
                 \node[below left] at (-0.02,0.05) {\footnotesize $0$};
                \node[below] at (1,-0.02) {\footnotesize $1$};
                \node[left] at (-0.02,1) {\footnotesize $1$};

                \node[below] at (0.25,-0.02) {\footnotesize $0.25$};
                \node[below] at (0.5,-0.02) {\footnotesize $0.5$};
                \node[below] at (0.75,-0.02) {\footnotesize $0.75$};

                \node[left] at (-0.02,0.25) {\footnotesize $0.25$};
                \node[left] at (-0.02,0.5) {\footnotesize $0.5$};
                \node[left] at (-0.02,0.75) {\footnotesize $0.75$};

            \end{scope}
        \end{scope}

    \end{tikzpicture}
    \centering
    \caption{Representation of $A_{\frac{1}{4}}$ for absolute distance.}
    \label{A}
\end{figure}

The final part of the paper is devoted to a consistent plug-in estimator for the newly introduced dependence functions $\phi_{(Y,X)}$ and $\kappa$. This estimator enables a detailed analysis of the underlying dependence structure. In particular, it provides explicit insights into the conditional distribution of $Y$ given $\XX$, including indications of conditional atomic or functional components. Moreover, we illustrate an application to the near-zero behavior of $\phi_{(Y,\XX)}$, where the estimator allows one to draw conclusions about potential conditional atomic components and structured dependence patterns.

\section{The dependence functions $\phi_{(Y,X)}$ and $\kappa$}

In the following section, we examine the dependence functions $\phi_{(Y,\XX)}$ and $\kappa_{(Y,\XX)}$ and investigate their fundamental properties under the assumption that $F_Y$ is continuous. This naturally raises the question which aspects of conditional dependence are captured by the function $\phi_{(Y,X)}$ and, in the subsequent section, by the function $\kappa$.

\subsection{Properties of \texorpdfstring{$\phi_{(Y,X)}$}{phi}}

For a random vector $(Y,\XX)$, the function $\phi_{(Y,\XX)}$ in \eqref{Phi1} is defined as the distribution function of the random variable
\[
Z = \lvert F_Y(Y) - F_Y(Y') \rvert,
\]
where $(Y,Y')$ is a version of the Markov product of $(Y,\XX)$. In particular, its distribution function provides a characterization of independence through Chatterjee’s correlation coefficient.
Thus, $\phi_{(Y,\XX)}(0)$ provides quantitative information about exact conditional ties in the Markov product, or equivalently about atomic components in the conditional distribution of $F_Y(Y)$ given $\XX$.

\begin{theorem}
Let $(Y,\XX)$ be a $(d+1)$-dimensional random vector.
\(
\phi_{(Y,\XX)} = \phi^{\perp}_{(Y,\XX)}
\)
if and only if  $Y$ and $\XX$ are independent.
\end{theorem}

\begin{proof}
Assume first that $\phi_{(Y,\XX)}=\phi^{\perp}_{(Y,\XX)}$. Then
\begin{align}
\kappa_{(Y,\XX)}(1)
&= 1 - 3 + 3\int_0^1 \phi_{(Y,\XX)}(s)\,\mathrm{d}s \notag\\
&= 3\int_0^1 \phi^{\perp}_{(Y,\XX)}(s)\,\mathrm{d}s -2 \notag\\
&= 3\int_0^1 (2s-s^2)\,\mathrm{d}s -2
=0. \notag
\end{align}
Since $\kappa_{(Y,\XX)}(1)=\xi(Y,\XX)$, independence follows from the characterization of $\xi$; see \cite{azadkia2021simple}.

Conversely, if $Y$ and $\XX$ are independent, then $V=F_Y(Y)$ and $V'=F_Y(Y')$ are independent uniform random variables on $[0,1]$ under the Markov product. Hence, for all $t\in[0,1]$,
\[
\PP(|V-V'|\le t)
=
1-\PP(|V-V'|>t)
=
1-(1-t)^2
=
2t-t^2.
\]
Therefore $\phi_{(Y,\XX)}=\phi^{\perp}_{(Y,\XX)}$.
\end{proof}
As noted in the introduction, $\phi_{(Y,\XX)}(0)$ reflects atomic mass in the conditional distribution of $F_Y(Y)$ given $\XX$.
\begin{theorem}
Let $(Y,\XX)$ be a $(d+1)$-dimensional random vector. The following statements are equivalent:
\begin{itemize}
\item[(i)] For $\mathbb{P}^{\XX}$-almost every $\xx\in\mathbb{R}^d$, the conditional distribution of $F_Y(Y)$ given $\XX=\xx$ has no atoms, i.e.,
\[
\mathbb{P}\bigl(F_Y(Y)=v \mid \XX=\xx\bigr)=0
\quad \text{for every } v\in[0,1].
\]
\item[(ii)] $\phi_{(Y,\XX)}(0)=0$.
\end{itemize}
\end{theorem}

\begin{proof}
Let $V=F_Y(Y)$ and $V'=F_Y(Y')$. By disintegration,
\begin{align}
\phi_{(Y,\XX)}(0)
&=\mathbb{P}(V=V') \notag\\
&=\int_{\mathbb{R}^d}
\mathbb{P}(V=V'\mid \XX=\xx)\,d\mathbb{P}^{\XX}(\xx) \notag\\
&=\int_{\mathbb{R}^d}
\left(
\int_{[0,1]}
\mathbb{P}(V=v\mid \XX=\xx)\,
d\mathbb{P}^{V\mid \XX=\xx}(v)
\right)
d\mathbb{P}^{\XX}(\xx). \label{t1}
\end{align}

The inner integral satisfies
\[
\int_{[0,1]}
\mathbb{P}(V=v\mid \XX=\xx)\,
d\mathbb{P}^{V\mid \XX=\xx}(v)
=
\sum_{a}\mathbb{P}(V=a\mid \XX=\xx)^2,
\]
where the sum runs over all atoms of the conditional distribution of $V$ given $\XX=\xx$.

Hence
\[
\phi_{(Y,\XX)}(0)
=
\int_{\mathbb{R}^d}
\sum_{a}\mathbb{P}(V=a\mid \XX=\xx)^2
\,d\mathbb{P}^{\XX}(\xx).
\]

Since the integrand is nonnegative, $\phi_{(Y,\XX)}(0)=0$ holds if and only if
\[
\sum_{a}\mathbb{P}(V=a\mid \XX=\xx)^2=0
\quad \text{for }\mathbb{P}^{\XX}\text{-almost every }\xx,
\]
which is equivalent to the conditional distribution of $V$ given $\XX=\xx$ being non-atomic. This proves the claim.
\end{proof}

The previous theorem characterizes the absence of exact conditional atomic components in the Markov product. The next result identifies the opposite extreme.

\begin{theorem}
Let $(Y,\XX)$ be a $(d+1)$-dimensional random vector. Then the following statements are equivalent:
\begin{itemize}
\item[(i)] $\phi_{(Y,\XX)}(0)=1$.
\item[(ii)] There exists a measurable function $f$ such that $Y=f(\XX)$ almost surely.
\end{itemize}
\end{theorem}

\begin{proof}
We first show \emph{(i) implies (ii)}. Assume that $\phi_{(Y,\XX)}(0)=1$. By definition,
\[
\phi_{(Y,\XX)}(0)
=\mathbb{P}\bigl(F_Y(Y)=F_Y(Y')\bigr).
\]
Hence $F_Y(Y)=F_Y(Y')$ almost surely. By \eqref{t1}, this implies that, for $\mathbb{P}^{\XX}$-almost every $\xx$, the conditional distribution of $F_Y(Y)$ given $\XX=\xx$ is degenerate. By the standard measurability of regular conditional distributions, there exists a measurable function $h$ such that $F_Y(Y)=h(\XX)$ almost surely.
Let $F_Y^{-1}(u):=\inf\{y\in\mathbb{R}:F_Y(y)\ge u\}$ denote the generalized inverse of $F_Y$. Since $F_Y$ is continuous, $Y=F_Y^{-1}(F_Y(Y))$ almost surely. Hence $Y=F_Y^{-1}(h(\XX))$ almost surely, which proves (ii).

That (ii) implies (i) again follows from \eqref{t1}.
\end{proof}

\subsection{Properties of \texorpdfstring{$\kappa$}{kappa}}

Having established the fundamental properties of $\phi_{(Y,\XX)}$, we now turn to the function
$\kappa_{(Y,\XX)}$. Owing to its integral representation, $\kappa_{(Y,\XX)}$ inherits several
structural properties from $\phi_{(Y,\XX)}$, while the integration step introduces additional
regularity. We begin by collecting basic analytic properties of $\kappa_{(Y,\XX)}$.

\begin{lem}\label{lemam2.4}
Let $(Y,\XX)$ be a $(d+1)$-dimensional random vector. Then the function $\kappa_{(Y,\XX)}:[0,1]\to[0,1]$ is convex, non-increasing, and Lipschitz continuous with Lipschitz constant~$3$.
\end{lem}

\begin{proof}
Let $t\in[0,1]$ and $t'\in[0,1]$ with $t\le t'$. Then
\begin{align*}
\kappa_{(Y,\XX)}(t')-\kappa_{(Y,\XX)}(t)
&= -3(t'-t) + 3\int_t^{t'} \phi_{(Y,\XX)}(s)\,\mathrm{d}s \\
&= -3\int_t^{t'} \bigl(1-\phi_{(Y,\XX)}(s)\bigr)\,\mathrm{d}s.
\end{align*}
Since $0\le \phi_{(Y,\XX)}(s)\le 1$ for all $s\in[0,1]$, the integrand is non-negative. Hence $\kappa_{(Y,\XX)}$ is non-increasing. Moreover,
\[
\left|\kappa_{(Y,\XX)}(t')-\kappa_{(Y,\XX)}(t)\right|
\le 3|t'-t|,
\]
so $\kappa_{(Y,\XX)}$ is Lipschitz continuous with Lipschitz constant~$3$.

Since $\phi_{(Y,\XX)}$ is non-decreasing, the function $1-\phi_{(Y,\XX)}$ is non-increasing. Therefore the slopes of $\kappa_{(Y,\XX)}$ are non-decreasing, which implies that $\kappa_{(Y,\XX)}$ is convex.

For the bounds, note that
\[
\kappa_{(Y,\XX)}(t)
= 1 - 3\int_0^t \bigl(1-\phi_{(Y,\XX)}(s)\bigr)\,\mathrm{d}s.
\]
Since $1-\phi_{(Y,\XX)}(s)\ge 0$, it follows that $\kappa_{(Y,\XX)}(t)\le 1$ for all $t\in[0,1]$. Further, as $\kappa_{(Y,\XX)}$ is non-increasing, its minimum is attained at $t=1$, and
\[
\kappa_{(Y,\XX)}(1)=\xi(Y,\XX)\ge 0.
\]
Thus $\kappa_{(Y,\XX)}(t)\in[0,1]$ for all $t\in[0,1]$.
\end{proof}

The monotonicity of $\kappa_{(Y,\XX)}$ allows us to characterize independence through the vanishing
of $\kappa_{(Y,\XX)}$ at any positive argument.

\begin{theorem}\label{ter2}
Let $(Y,\XX)$ be a $(d+1)$-dimensional random vector. The following statements are equivalent:
\begin{itemize}
\item[(i)] There exists some $t\in(0,1]$ such that $\kappa_{(Y,\XX)}(t)=0$.
\item[(ii)] $\kappa_{(Y,\XX)}(1)=0$.
\item[(iii) ] $Y$ and $\XX$ are independent.
\item[(iv)] $\kappa_{(Y,\XX)}(1)=\kappa^{\perp}_{(Y,\XX)}(1)$.
\end{itemize}
\end{theorem}

\begin{proof}
We first show that (i) implies (ii).
Since $\kappa_{(Y,\mathbf{X})}$ is non-increasing, Lemma~\ref{lemam2.4} together with the lower bound $\kappa_{(Y,\XX)}\ge 0$ implies that the existence of some $t>0$ such that
$\kappa_{(Y,\mathbf{X})}(t)=0$ yields
\[
\kappa_{(Y,\mathbf{X})}(1)=0.
\]

Next we show that \emph{(ii) implies (iii).}
By definition, $\kappa_{(Y,\XX)}(1)=\xi(Y,\XX)$. Since $\xi(Y,\XX)=0$ characterizes independence,
it follows that $Y$ and $\XX$ are independent.

Next we show that \emph{(iii) implies (iv).}
If $Y$ and $\XX$ are independent, then
\[
\kappa_{(Y,\XX)}(1)=\kappa^{\perp}_{(Y,\XX)}(1).
\]

Finally we show that \emph{(iv) implies (i).}
Since $\kappa^{\perp}_{(Y,\XX)}(1)=0$, equality implies $\kappa_{(Y,\XX)}(1)=0$. Choosing $t=1$
yields $\kappa_{(Y,\XX)}(t)=0$, which proves~(i).
\end{proof}

Finally, we consider the opposite extreme and show that maximal values of $\kappa_{(Y,\XX)}$
correspond to perfect dependence.

\begin{theorem}
Let $(Y,\XX)$ be a $(d+1)$-dimensional random vector. Then the following are equivalent:
\begin{enumerate}
\item[(i)] There exists $t\in(0,1]$ such that $\kappa_{(Y,\XX)}(t)=1$.
\item[(ii)] $Y$ is perfectly dependent on $\XX$.
\end{enumerate}
\end{theorem}

\begin{proof}
Let $V=F_Y(Y)$ and $V'=F_Y(Y')$.

We first show that \emph{(i) implies (ii).}
If $\kappa_{(Y,\XX)}(t)=1$ for some $t\in(0,1]$, then, by the definition of $\kappa_{(Y,\XX)}$,
\[
\int_0^t \bigl(1-\phi_{(Y,\XX)}(s)\bigr)\,\mathrm{d}s=0.
\]
Since $0\le \phi_{(Y,\XX)}(s)\le 1$ for all $s\in[0,1]$, this implies
\[
\phi_{(Y,\XX)}(s)=1
\quad \text{for Lebesgue-a.e. } s\in[0,t].
\]
Because $\phi_{(Y,\XX)}$ is non-decreasing and right-continuous, it follows that
\[
\phi_{(Y,\XX)}(s)=1
\quad \text{for all } s\in(0,t].
\]
By right-continuity at $0$, we obtain $\phi_{(Y,\XX)}(0)=1$, and hence
\[
\mathbb{P}(V=V')=1.
\]

Since $V$ and $V'$ are conditionally independent and identically distributed given $\XX$, this means that, for $\mathbb{P}^{\XX}$-almost every $\xx$, two independent draws from the conditional distribution of $V$ given $\XX=\xx$ are equal almost surely. Therefore this conditional distribution is degenerate. Hence there exists a measurable function $h$ such that
\[
V=h(\XX)
\quad \text{almost surely}.
\]
Since $F_Y$ is continuous, we have
\[
Y=F_Y^{-1}(F_Y(Y))=F_Y^{-1}(V)
\quad \text{almost surely},
\]
where $F_Y^{-1}(u):=\inf\{y\in\mathbb{R}:F_Y(y)\ge u\}$. Thus, with
\[
f:=F_Y^{-1}\circ h,
\]
we obtain
\[
Y=f(\XX)
\quad \text{almost surely}.
\]
Hence $Y$ is perfectly dependent on $\XX$.

Conversely, assume that $Y$ is perfectly dependent on $\XX$, i.e. $Y=f(\XX)$ almost surely for some measurable function $f$. Then the conditional distribution of $Y$ given $\XX$ is degenerate. Hence, under the Markov product,
\[
Y=Y'
\quad \text{almost surely},
\]
and therefore
\[
F_Y(Y)=F_Y(Y')
\quad \text{almost surely}.
\]
It follows that $\phi_{(Y,\XX)}(s)=1$ for all $s\in[0,1]$. Consequently,
\[
\kappa_{(Y,\XX)}(t)
= 1 - 3\int_0^t \bigl(1-\phi_{(Y,\XX)}(s)\bigr)\,\mathrm{d}s
= 1
\]
for all $t\in(0,1]$.
\end{proof}

\subsection{Examples}

Before turning to the examples, we briefly recall the basic concepts of copulas and the associated Markov product. 
Consider a $(d+1)$-dimensional random vector $(Y,\XX)$ with continuous marginal distribution functions $F_i$ of $X_i$, $i \in \{1,\ldots,d\}$, and $F_Y$ of $Y$, and connecting copula $C$. Denoting
\[
U_i := F_{X_i}(X_i), \quad i \in \{1,\ldots,d\}, \qquad V := F_Y(Y),
\]
we have $(\mathbf{U},V) \sim C$. We extend the random vector $(\mathbf{U},V)$ and consider the $(d+2)$-dimensional random vector $(\mathbf{U},V,V')$, where $V$ and $V'$ share the same conditional distribution and are conditionally independent given $\mathbf{U}$. Using disintegration, the distribution function of $(V,V')$ is a copula.

Let $\mathcal{C}^k$ denote the space of $k$-dimensional copulas. Then the map $\psi : \mathcal{C}^{d+1} \to \mathcal{C}^2$, defined by
\[
\psi(C)(s,t)
:= \PP(V \le s, V' \le t), \qquad (s,t) \in [0,1]^2,
\]
transforms every $(d+1)$-dimensional copula into an exchangeable bivariate copula, see \cite{fuchs2024JMVA}.

\begin{example}([Fréchet copula]\cite{fuchs2024JMVA}\label{fre})
For \(\alpha,\beta \in [0,1]\) with \(\alpha + \beta \leq 1\), the Fréchet copula is defined by
\[
C_{\alpha, \beta} := \alpha M + (1 - \alpha - \beta) \Pi + \beta W,
\]
where $W$ denotes the lower Fr\'echet--Hoeffding bound (see \cite{durante2015principles}), and the Markov product of $C_{\alpha,\beta}$ is given by
\[
\psi(C_{\alpha,\beta}) = C_{\alpha^2 + \beta^2, 2 \alpha \beta}.
\]
We obtain
\begin{align*}
\phi_{(Y,\XX)}(t) &= (\alpha^2 + \beta^2) + 2 (1 - (\alpha^2 + \beta^2) - \alpha \beta) t - (1 - (\alpha + \beta)^2) t^2,
\end{align*}
\begin{align*}
\kappa_{(Y,\XX)}(t) &= 1+3 (\alpha^2 + \beta^2-1) t + 3 (1 - (\alpha^2 + \beta^2) - \alpha \beta) t^2 - (1 - (\alpha + \beta)^2) t^3 .
\end{align*}
For visualization, see Figure~\ref{fig:combined}. The jump of $\phi_{(Y,\mathbf{X})}(0)=\alpha^2+\beta^2$ at $0$ is due to the $\alpha^2$ contribution through $M$ and the $\beta^2$ contribution through $W$. Both effects reflect functional relationships, and the quadratic terms arise from the construction using $Y$ and an independent copy $Y’$. In this sense, the resulting structure—namely, the mass concentrated on the diagonal—does not depend on our choice of measurement function but is already inherent in the Markov product.

In particular, we obtain $\kappa_{(Y,\mathbf{X})}(1)=\alpha^2-\alpha\beta+\beta^2.$
An interesting feature of the endpoint value $\kappa_{(Y,\mathbf{X})}(1)$ is that the mass not concentrated on the diagonal, that is, the independent component $\Pi$, is neutral in this particular scalar expression. What is noteworthy here is that neither $\phi_{(Y,X)}$ in general nor $\kappa$---even in the explicit case $t=1$---exhibits a monotone increase of the dependence measure when $\alpha$ or $\beta$ is increased. This can be illustrated by considering $\kappa$ at $t=1$:
\[
\frac{\partial}{\partial \alpha}\kappa_{(Y,\XX)}(1)=2\alpha-\beta,
\qquad
\frac{\partial}{\partial \beta}\kappa_{(Y,\XX)}(1)=2\beta-\alpha.
\]

\begin{figure}[ht]
    \centering
    \includegraphics[width=0.8\textwidth]{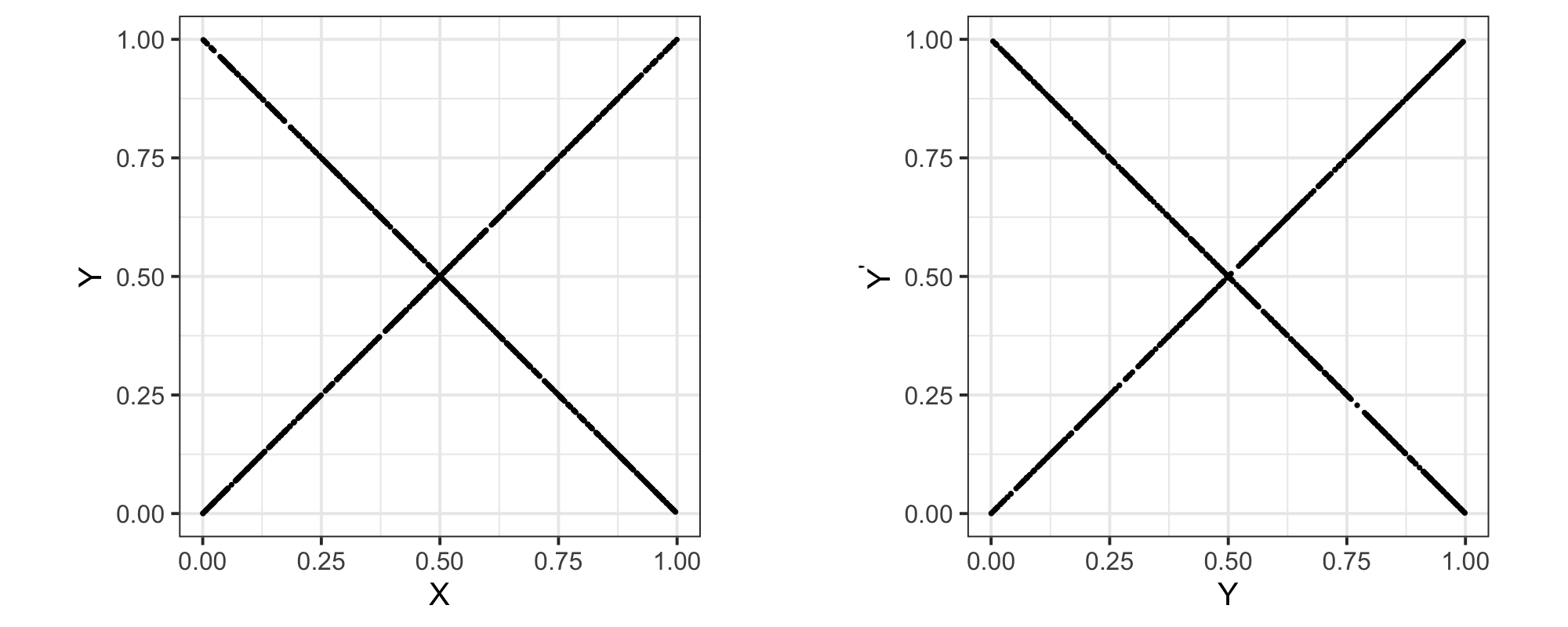}

    \vspace{0.8em}

    \includegraphics[width=0.8\textwidth]{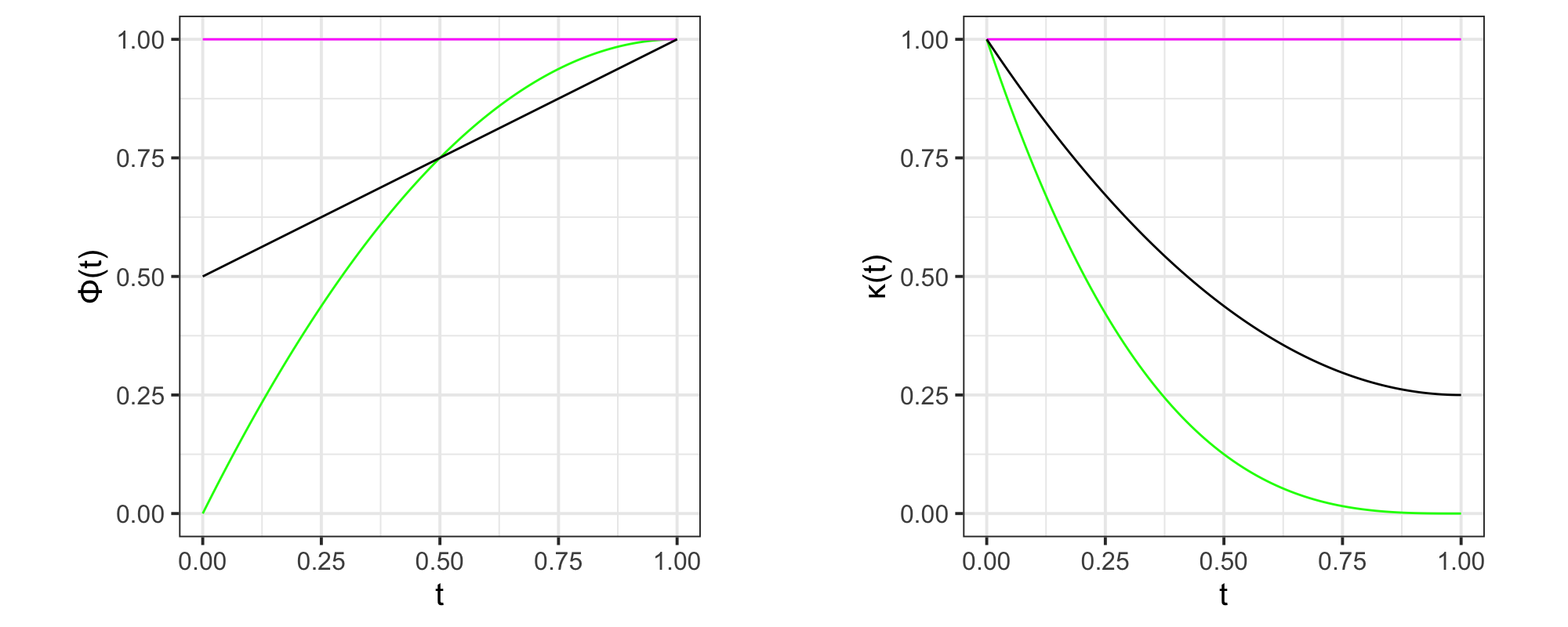}

   \caption{Example \ref{fre}; First line: Scatterplot with $n = 1000$ of a random vector $(Y,X)$ associated with a Fréchet copula with parameters $\alpha=\beta=0.5$ and its Markov product. Second line: Visualization of $\phi_{(Y,X)}$ and $\kappa_{(Y,X)}$.}
\label{fig:combined}
\end{figure}

\end{example}

\begin{example}(\label{F_Yaus}[Gaussian]\cite{fuchs2024JMVA})
Let $C_\rho$ be an equicorrelated $(d+1)$-dimensional Gaussian copula with parameter $\rho$.
From \cite{fuchs2024JMVA}, we know for $\rho \in \left(\frac{-1}{d},1\right)$ that $\psi(C_\rho)$ corresponds to a bivariate Gaussian copula with correlation parameter
\[
\rho^*(d) = \frac{d\,\rho^2}{1 + (d-1)\rho}.
\]
For $\rho \in (0,1)$, we have $\rho^*(d)<\rho$ and
$\rho^*(d)<\rho^*(d')$ for all $d,d'\in\mathbb{N}$ with $d<d'$.
The expressions for $\phi_{(Y,\mathbf{X})}(t)$ and $\kappa_{(Y,\mathbf{X})}(t)$ can be computed explicitly, but they do not admit a meaningful simplification. We therefore omit their analytical forms and instead refer to the visualization in Figure~\ref{fig:gaussian_combined}. We observe that, as $d$ increases, $\phi_{(Y,\XX)}(t)$ and $\kappa_{(Y,\XX)}(t)$ increase monotonically in $d$ for all $t\in(0,1]$. In this setting, $\phi_{(Y,\mathbf{X})}(0)=0$, since $\mathbb{P}(Y=y\mid\mathbf{X})=0$ for all $y\in\mathbb{R}$.

In the Gaussian equicorrelated setting, Chatterjee's rank correlation of $C_\rho$
coincides with Spearman's footrule of $\psi(C_\rho)$; see \cite{ansari2025continuity,fuchs2024JMVA}.
Hence, using the classical arcsine representation for Gaussian copulas
\cite{nelsen2006}, we obtain
\[
\kappa_{(Y,\XX)}(1)
=
\frac{3}{\pi}\arcsin\!\left(
\frac{1+\rho^*(d)}{2}
\right)-\frac{1}{2}
=
\frac{3}{\pi}\arcsin\!\left(
\frac{1+(d-1)\rho+d\rho^2}{2(1+(d-1)\rho)}
\right)-\frac{1}{2}.
\]

\begin{figure}[ht]
    \centering
    \includegraphics[width=0.9\textwidth]{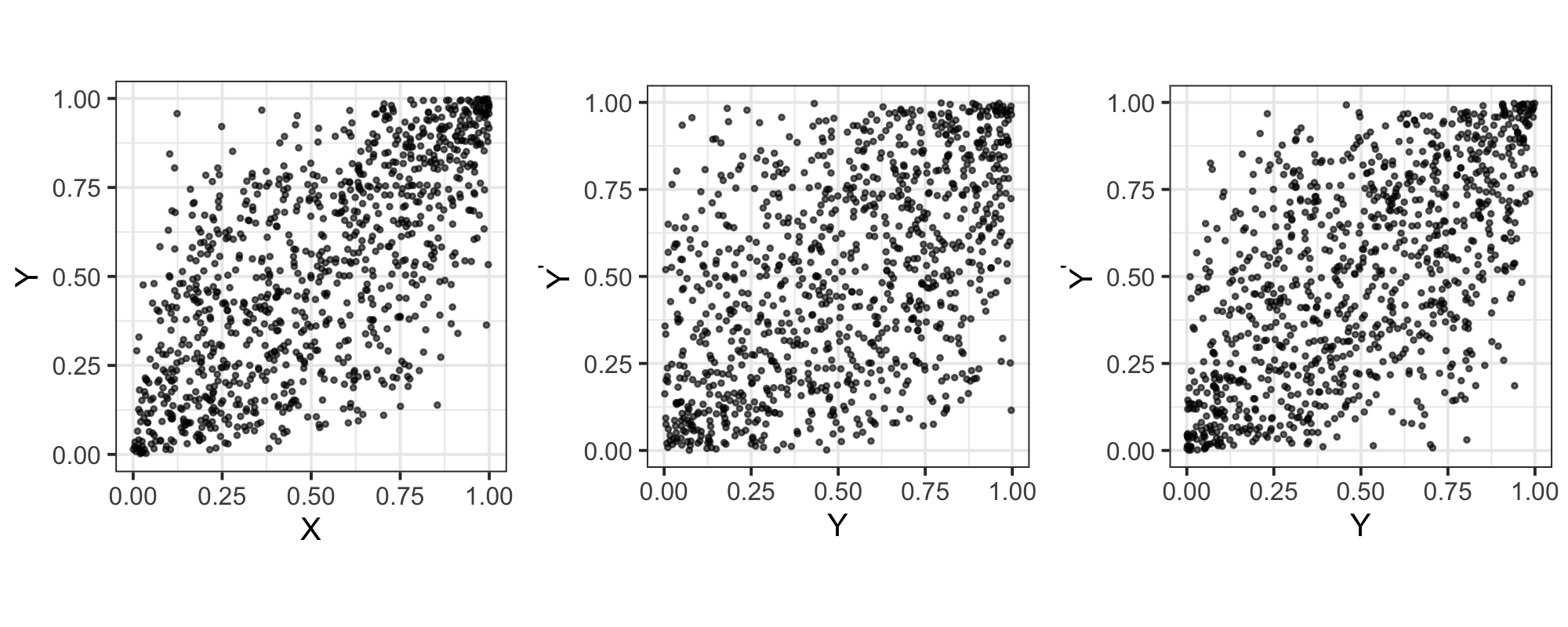}

    \vspace{0.8em}

    \includegraphics[width=0.8\textwidth]{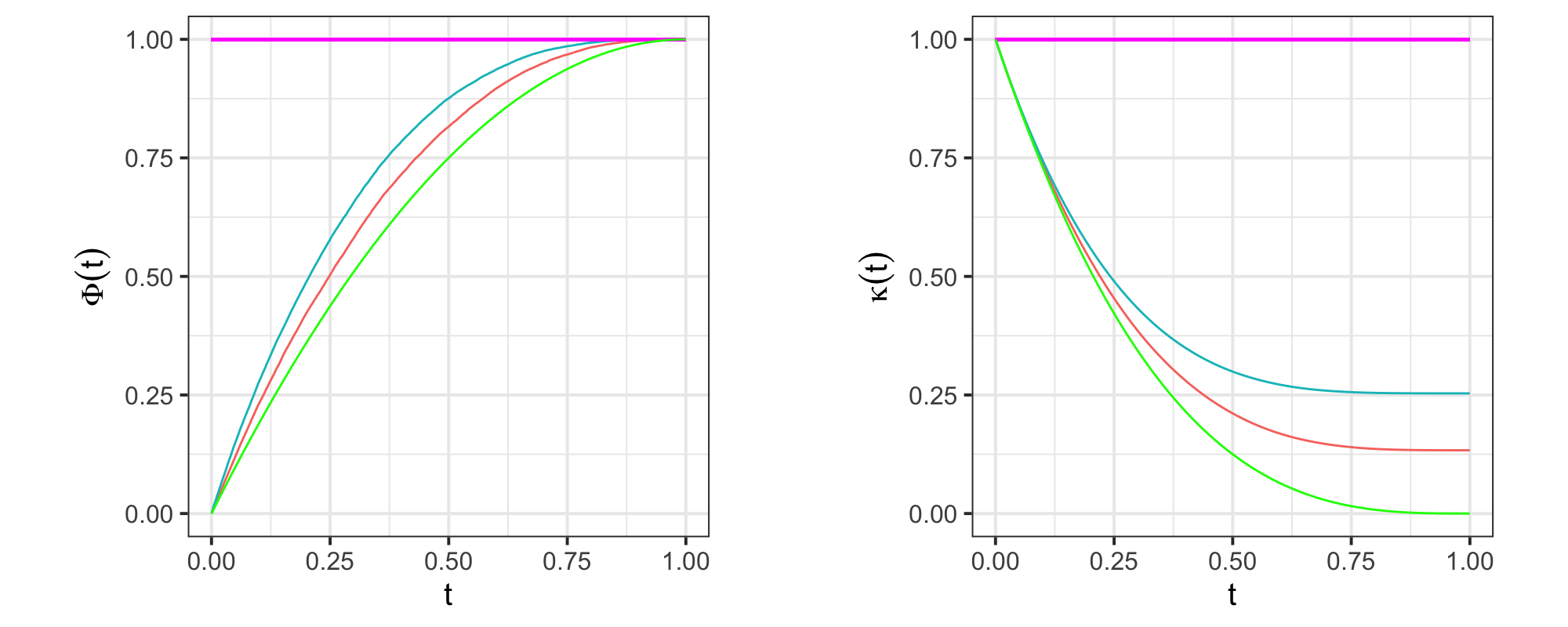}

\caption{%
Gaussian copula example with base correlation parameter
\(\rho=\sqrt{0.5}\). The first row shows simulated samples
(\(n = 1000\)): a bivariate Gaussian copula sample and the corresponding
Markov-product copulas \(C_{\rho^*(d)}\), with
\(\rho^*(1)=\tfrac12\) for \(d=1\) and
\(\rho^*(5)=\tfrac{5}{2+4\sqrt2}\) for \(d=5\).
The second row displays the corresponding functions
\(\phi_{(Y,\mathbf X)}\) and \(\kappa_{(Y,\mathbf X)}\) for the
equicorrelated Gaussian copula with \(d=1\) and \(d=5\).}
    \label{fig:gaussian_combined}
\end{figure}
\end{example}

\begin{example}([Marshall--Olkin copula]\cite{fuchs2024JMVA}\label{Olkin})
Let $\alpha,\beta \in [0,1]$. The Marshall--Olkin copula is defined by
\[
C_{\alpha,\beta}(u,v) := \min\{u^{1-\alpha}v,\; uv^{1-\beta}\}.
\]
It holds that $C_{\alpha,\beta} = \Pi$, where $\Pi$ denotes the independence copula (see \cite{durante2015principles}), if and only if $\min\{\alpha,\beta\} = 0$, and that $C_{\alpha,\beta} = M$, the upper Fr\'echet--Hoeffding copula (see \cite{durante2015principles}), if and only if $\min\{\alpha,\beta\} = 1$.

In the following, we focus on the case $\alpha = 1$. On the one hand, this choice keeps us within the same copula class; on the other hand, it is particularly illustrative, as shown in Figure~\ref{fig:MO_combined}. In this setting, the copula exhibits a functional component with positive mass, whose graph is given by $(u, u^{1/\beta})$. This structure can be detected, since the Markov product projects it onto the diagonal, where it is then captured analytically by the value $\phi_{(Y,\mathbf{X})}(0)$. Thus
\[
\psi(C_{1,\beta})(s,t)=C_{\beta,\beta}(s,t).
\]
Conditionally on $U=u$, the functional component induces an atom at $u^{1/\beta}$ of size
\[
u^{(1-\beta)/\beta}.
\]
The squared expectation of this atom size is
\[
\mathbb{E}\!\left[U^{2(1-\beta)/\beta}\right]
=
\int_0^1 u^{2(1-\beta)/\beta}\,\mathrm{d}u
=
\frac{\beta}{2-\beta}.
\]
Hence,
\[
\phi_{(Y,\mathbf{X})}(0)
= \frac{\beta}{2-\beta},
\qquad
\kappa_{(Y,\mathbf{X})}(1)
= \frac{2\beta}{3-\beta}.
\]
We thus observe that functional relationships are transformed into mass on the diagonal, provided that for each $\xx$ the conditional distribution of $F_Y(Y)\mid \XX=\xx$ has at most one atom.

\begin{figure}[ht]
    \centering
    \includegraphics[width=0.8\textwidth]{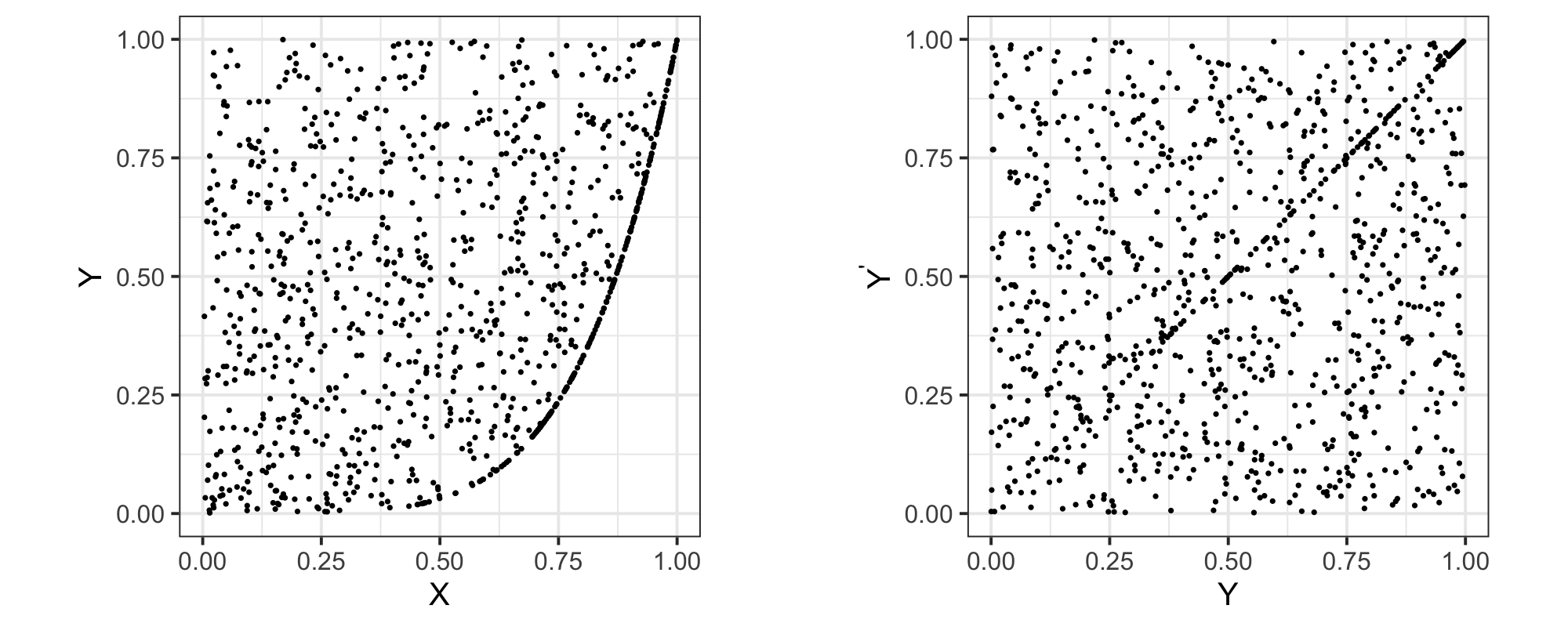}

    \vspace{0.8em}

    \includegraphics[width=0.8\textwidth]{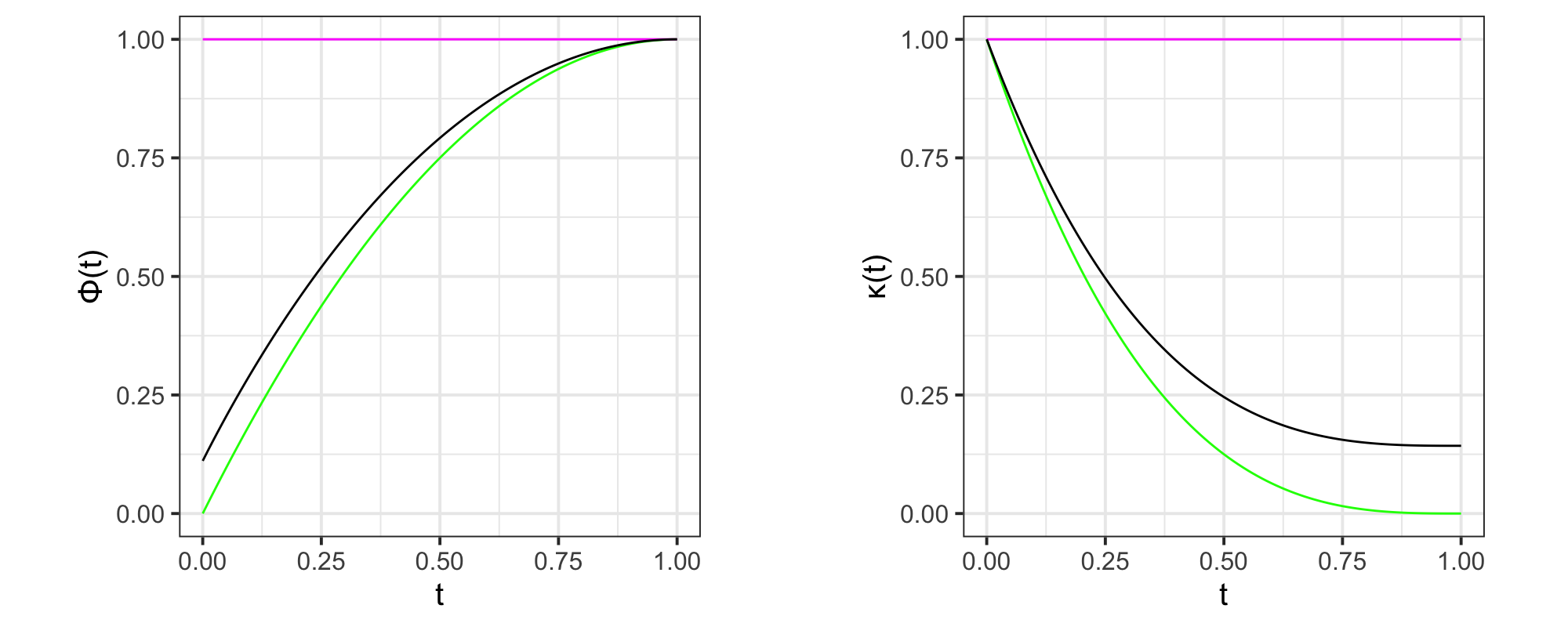}

    \caption{%
    First row: Samples of size $n = 1000$ drawn from a Marshall--Olkin copula with
    parameter vector $(\alpha,\beta)=(1,0.2)$ and its Markov product.
    Second row: Visualization of $\phi_{(Y,\mathbf{X})}$ and $\kappa_{(Y,\mathbf{X})}$
    for the same Marshall--Olkin copula.}
    \label{fig:MO_combined}
\end{figure}
\end{example}

\begin{example}([Shuffle copula and convex combination]\label{Jump})
Let \( h \in \mathbb{R} \), and define the bivariate shuffle copula $C_{h}$ as the copula obtained via \( f(x) = (x + h) \mod 1 \). For $m\in\mathbb{N}$, let
\[
C_m := \frac{1}{2^m}\sum_{i=1}^{2^m} C_{\frac{i}{2^m}},
\]
where $C_{i/2^m}$ denotes the shuffle copula generated by the shift $h=i/2^m$. This is the Jump copula with parameter $m\in \mathbb{N}$ and constitutes a special case of a convex combination of shifted copulas. We obtain:
\[
\psi(C_m) = C_m.
\]
Closed-form expressions of $\phi_{(Y,\XX)}$ and $\kappa_{(Y,\XX)}$ do not admit meaningful simplifications, so we focus on their graphical illustration; see Figure~\ref{fig:shuffle_combined}.

\[
\phi_{(Y,\XX)}(0)
= \frac{1}{2^m},
\qquad
\kappa_{(Y,\XX)}(1)
= \frac{1}{2^{2m}}.
\]
This class of copulas is relevant in several respects. An increase in the degree of dependence induced by conditioning on an additional variable does not necessarily lead to an increase of $\phi_{(Y,X)}$ in general; see Figure~\ref{fig:shuffle_combined}. Furthermore, the class is idempotent with $\lim_{m\to\infty} C_m=\Pi$ in distribution, and provides an explicit illustration of how the Markov product preserves and transforms dependence structures; see \cite{durante2015principles}.

\begin{figure}[ht]
    \centering
    \includegraphics[width=0.8\textwidth]{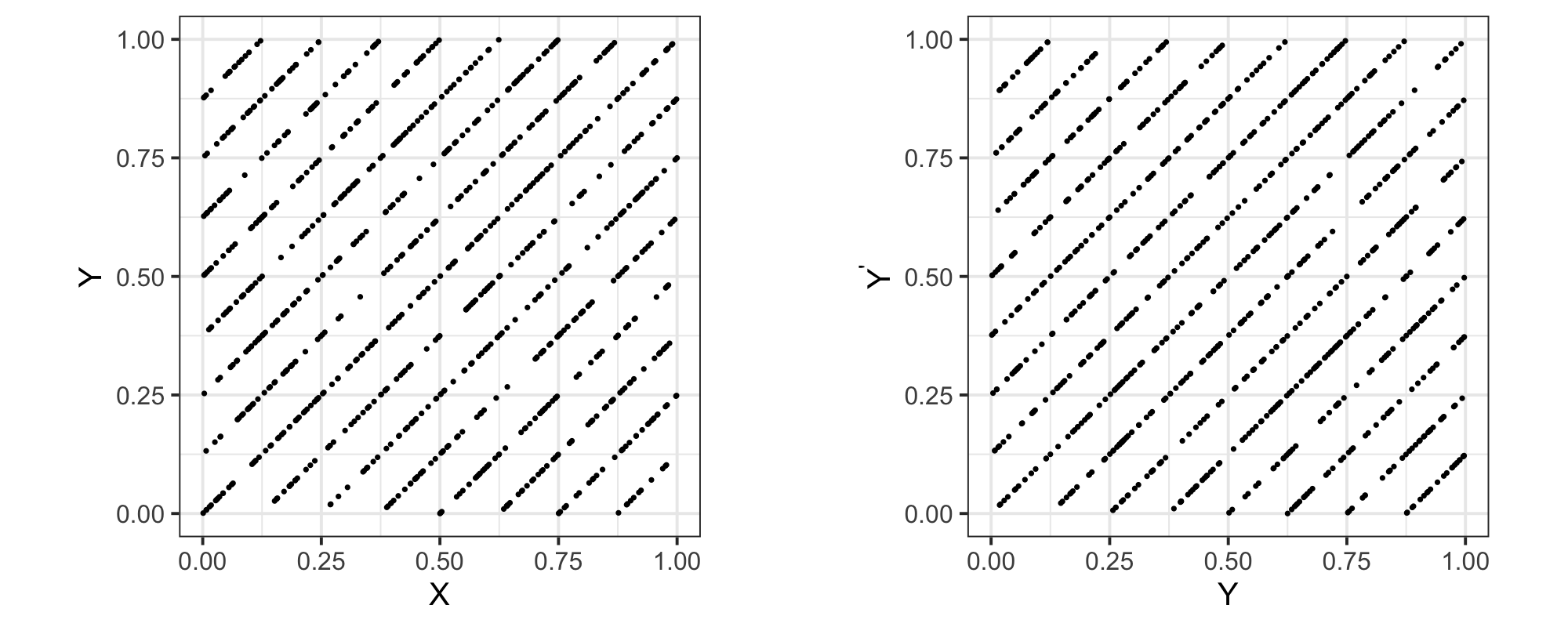}

    \vspace{0.8em}

    \includegraphics[width=0.8\textwidth]{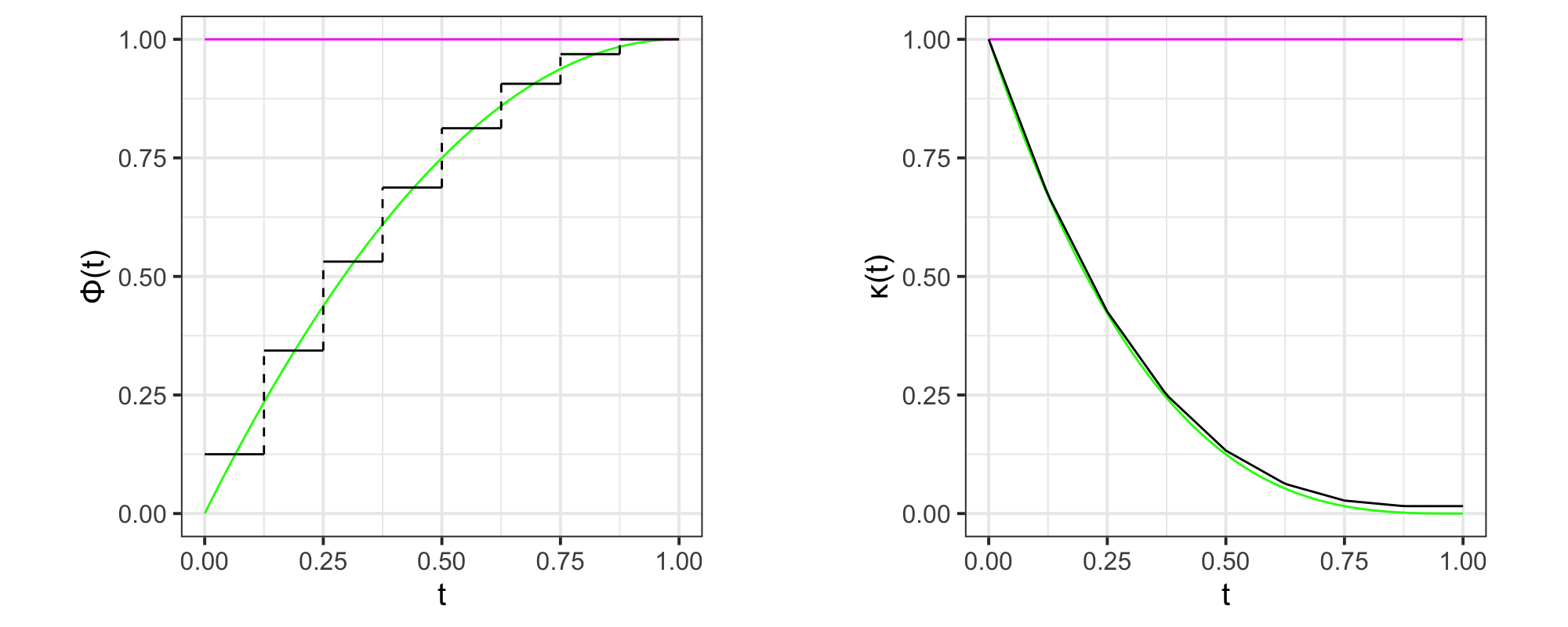}

\caption{%
Top: Scatterplot of the Jump copula $C_m$ with sample size $n = 1000$ and its Markov product
for $m=3$.
Bottom: Corresponding plots of $\phi_{(Y,\mathbf{X})}$ and $\kappa_{(Y,\mathbf{X})}$
for Example~\ref{Jump}.}
    \label{fig:shuffle_combined}
\end{figure}

\end{example}

\begin{example}([LSL-copula] \cite{fuchs2026semilinear}\label{LSLexample})
Let $\delta \colon [0,1]\to[0,1]$ be a copula diagonal, that is, a non-decreasing,
$2$-Lipschitz continuous function satisfying $\delta(t)\le t$ for all $t\in[0,1]$
with
\[
t \mapsto \frac{\delta(t)}{t} \ \text{non-decreasing},
\qquad
t \mapsto \frac{\delta(t)}{t^2} \ \text{non-increasing}.
\]
Denote by $\mathcal{D}^{\mathrm{LSL}}$ the class of all such functions.
For $\delta\in\mathcal{D}^{\mathrm{LSL}}$, the associated lower semilinear copula
$C_\delta\colon[0,1]^2\to[0,1]$ is given by
\[
C_\delta(x,y)=
\begin{cases}
y\,\dfrac{\delta(x)}{x} & y\le x, \\[1.2ex]
x\,\dfrac{\delta(y)}{y} & \text{otherwise}.
\end{cases}
\]
Then
\[
\operatorname{sing}(C_{\delta})
=
2\int_{[0,1]}\frac{\delta(x)}{x}\,d\lambda(x)-1
\]
denotes the mass of the singular part of the copula \(C_\delta\).
The Markov product preserves the LSL structure and satisfies
\[
\psi(C_\delta)=C_{\delta^*},
\]
where
\[
\delta^*(x)
= \frac{\delta(x)^2}{x}
+ x^2 \int_x^1 \left( \frac{\delta'(u)}{u} \right)^2 du,
\]
where $\delta'$ denotes the a.e. derivative of the Lipschitz
function $\delta$.
Closed-form expressions for $\phi_{(Y,\XX)}$ and $\kappa_{(Y,\XX)}$
are in general not tractable, and we therefore focus on their graphical
representation; see Figure~\ref{fig:LSL_combined}.
In particular, we observe that
\[
\phi_{(Y,\XX)}(0)=\operatorname{sing}(C_{\delta^*})=
-1
+ 2\int_0^1 \frac{\delta(u)^2}{u^2}\,du
+ \int_0^1 \bigl(\delta'(u)\bigr)^2\,du 
\]
and
\[
\kappa_{(Y,\XX)}(1)
=
\frac{2 \, \tau(C_\delta)^2}{1+\tau(C_\delta)},
\]
where $\tau$ denotes Kendall's tau. This identity follows from the lower semilinear copula results in \cite{fuchs2026semilinear}.

\begin{figure}[ht]
    \centering
    \includegraphics[width=0.85\textwidth]{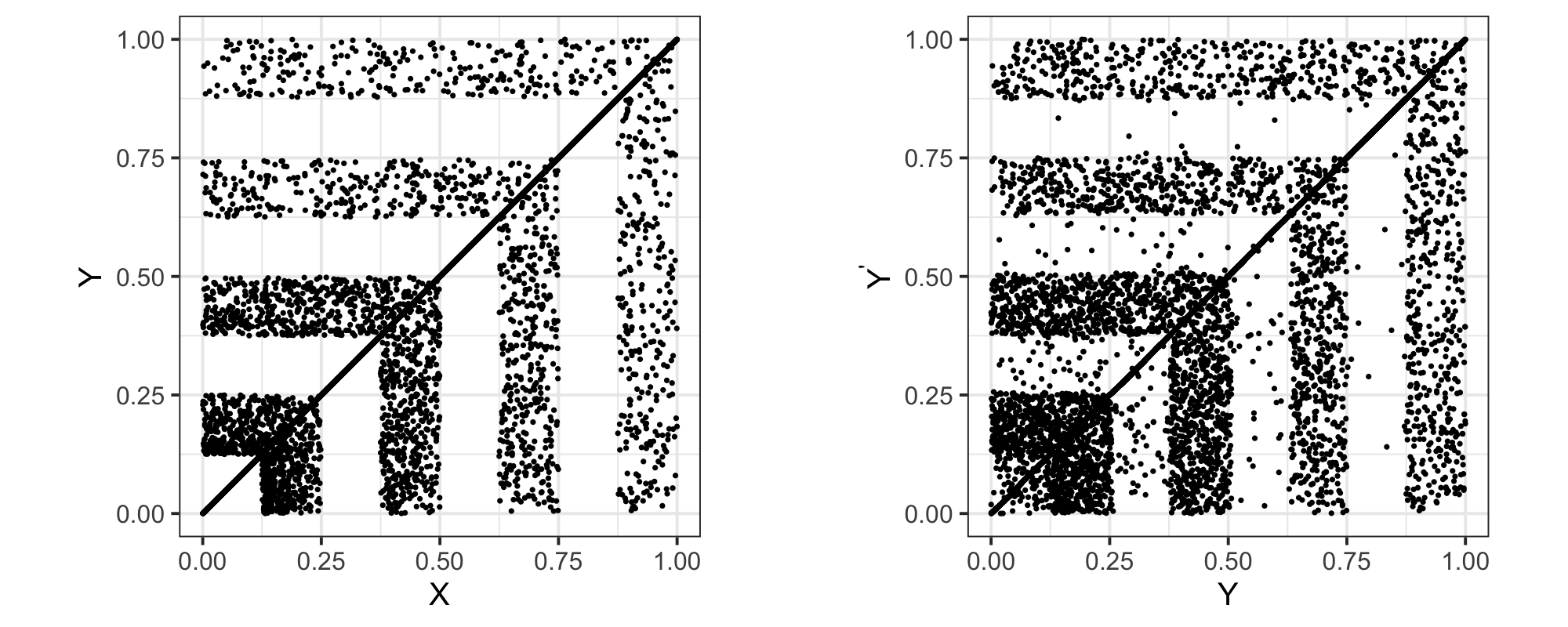}

    \vspace{0.8em}

    \includegraphics[width=0.85\textwidth]{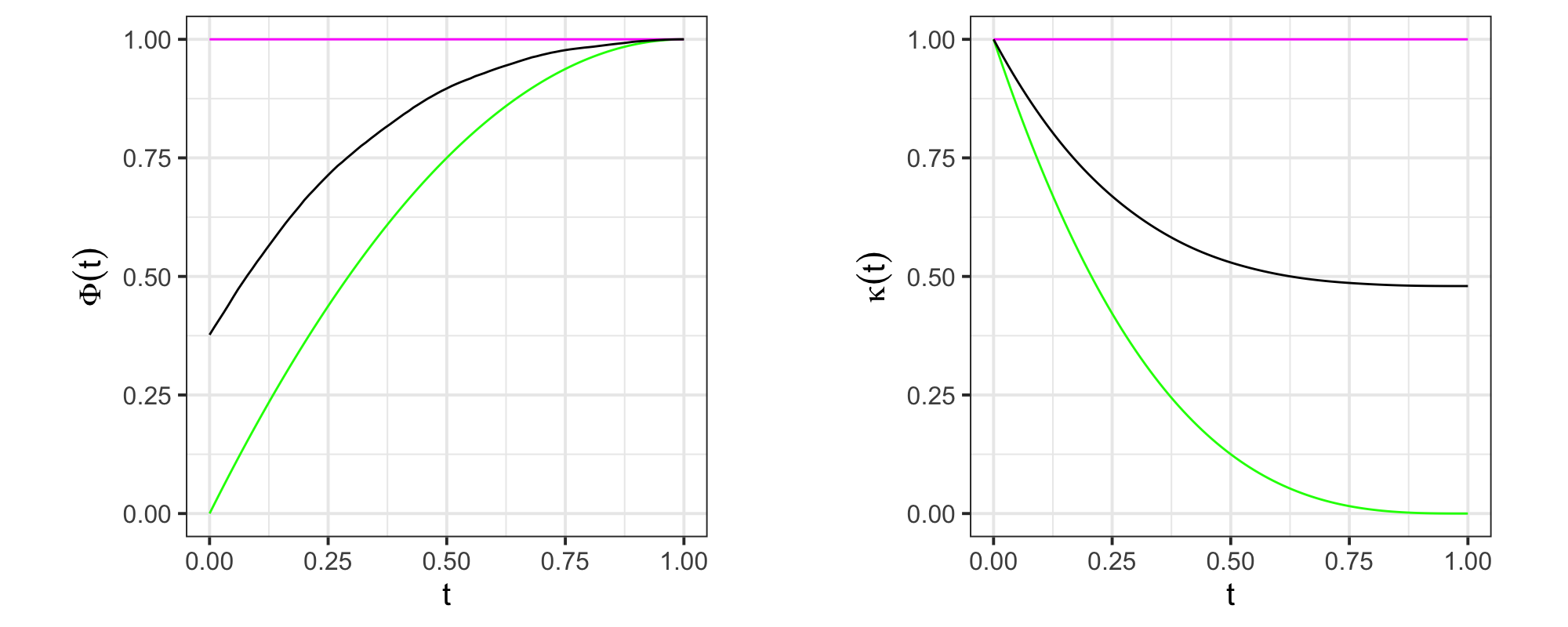}

    \caption{%
    Top: Samples of size $n = 10000$ drawn from a lower semilinear copula with
    $\delta$ given by linear interpolation of fixed support points, together with
    its Markov product.
    Bottom: Visualization of $\phi_{(Y,\mathbf{X})}$ and $\kappa_{(Y,\mathbf{X})}$
    for Example~\ref{LSLexample}.}
    \label{fig:LSL_combined}
\end{figure}
\end{example}

\section{Estimation} 

In this section, we study the asymptotic behavior of the estimators 
\(\hat{\phi}_{(Y,\mathbf X)}\) and \(\hat{\kappa}_{(Y,\mathbf X)}\). 
The general idea is closely related to \cite{fuchs2024JMVA}: 
the underlying Markov product distribution is approximated via a nearest-neighbor construction, which leads to a convenient stochastic representation of the estimator.

Let \(R_i\) denote the rank of \(Y_i\) among \(Y_1,\dots,Y_n\), where ties are resolved either by a fixed deterministic rule (for example, lexicographic ordering) or randomly. Furthermore, let \(N(i)\) denote the index of a nearest neighbor of \(\mathbf X_i\) among \(\{\mathbf X_j : j \neq i\}\) with respect to the Euclidean norm, where ties are again resolved according to a fixed rule.

Assume throughout that the conditions of Theorem~7 in \cite{fuchs2024JMVA}
are satisfied. In particular, \(F_Y\) is continuous,
\((Y_i,\mathbf X_i)_{i\ge 1}\) are i.i.d. observations with values in
\(\mathbb R\times\mathbb R^d\), nearest neighbors are constructed with respect
to the Euclidean norm on \(\mathbb R^d\), ties in ranks and nearest-neighbor
selection are treated according to the convention specified in
Theorem~7 of \cite{fuchs2024JMVA}, and the distribution of \(\mathbf X\)
satisfies the regularity assumptions required there for the almost sure weak
convergence of the empirical nearest-neighbor measure. All consistency
statements below are understood under these assumptions.

For \(t \in [0,1]\), define
\[
\hat{\phi}_{(Y,\mathbf{X})}(t)
:=
\frac{1}{n}\sum_{i=1}^n 
\mathbf{1}_{\left\{
\left|
\frac{R_i - R_{N(i)}}{n+1}
\right|
\le t
\right\}}.
\]

We first establish pointwise consistency of \(\hat{\phi}_{(Y,\mathbf X)}(t)\) at continuity points of the limit. We then strengthen this result to \(L^1\)-consistency. Finally, we discuss the boundary behavior at \(t=0\).

\begin{theorem}\label{thm:phi_pointwise}
Let $(Y,\mathbf X)$ be a $(d+1)$-dimensional random vector with copula $C$, and let
$(Y_1,\mathbf X_1),\dots,(Y_n,\mathbf X_n)$ be i.i.d.\ copies. Then, for every
$t\in[0,1]$ such that
\[
\PP(|V-V'|=t)=0,
\]
we have
\[
\hat{\phi}_{(Y,\mathbf X)}(t) \xrightarrow{a.s.} \phi_{(Y,\mathbf X)}(t),
\qquad (n\to\infty),
\]
where $V=F_Y(Y)$ and $V'=F_Y(Y')$.
\end{theorem}

\begin{proof}
Define the empirical nearest-neighbor measure on $[0,1]^2$ by
\[
\mu_n
:=
\frac1n\sum_{i=1}^n
\delta_{\left(\frac{R_i}{n+1},\,\frac{R_{N(i)}}{n+1}\right)}.
\]
By Theorem~7 in \cite{fuchs2024JMVA}, we have weak convergence
\[
\mu_n \xrightarrow{a.s.} \mu := \mathcal L(V,V').
\]
Let
\[
A_t := \{(u,v)\in[0,1]^2 : |u-v|\le t\}.
\]
Then $\hat{\phi}_{(Y,\mathbf X)}(t)=\mu_n(A_t)$ and $\phi_{(Y,\mathbf X)}(t)=\mu(A_t)$.
Since $\partial A_t=\{(u,v):|u-v|=t\}$, the condition $\PP(|V-V'|=t)=0$ implies $\mu(\partial A_t)=0$. Hence, by the Portmanteau theorem,
\[
\mu_n(A_t)\xrightarrow{a.s.}\mu(A_t).
\]
\end{proof}

This can now be exploited to establish $L^1$ convergence in a straightforward manner.

\begin{theorem}\label{thm:phi_L1}
Let $(Y,\mathbf X)$ be a $(d+1)$-dimensional random vector with copula $C$, and let
$(Y_1,\mathbf X_1),\dots,(Y_n,\mathbf X_n)$ be i.i.d.\ copies. Then
\[
\int_0^1 \bigl|\hat{\phi}_{(Y,\mathbf X)}(t)-\phi_{(Y,\mathbf X)}(t)\bigr|\,dt
\xrightarrow{a.s.} 0.
\]
\end{theorem}

\begin{proof}
By Theorem~\ref{thm:phi_pointwise}, we have
\[
\hat{\phi}_{(Y,\mathbf X)}(t)\to \phi_{(Y,\mathbf X)}(t)
\quad\text{a.s.}
\]
for every \(t\in[0,1]\) such that
\[
\PP(|V-V'|=t)=0.
\]

Now the function
\[
t\mapsto \PP(|V-V'|\le t)
\]
is a distribution function and hence monotone increasing. Therefore it has at most countably
many discontinuities. Since its discontinuities are precisely the points \(t\) for which
\[
\PP(|V-V'|=t)>0,
\]
the set
\[
D:=\{t\in[0,1]:\PP(|V-V'|=t)=0\}
\]
has full Lebesgue measure.

For \(t\in D\), Theorem~\ref{thm:phi_pointwise} yields
\[
\hat{\phi}_{(Y,\mathbf X)}(t)
\to
\phi_{(Y,\mathbf X)}(t)
\quad\text{almost surely}.
\]
Define
\[
B
:=
\left\{
(\omega,t)\in\Omega\times[0,1]:
\hat{\phi}_{(Y,\mathbf X)}(t,\omega)
\not\to
\phi_{(Y,\mathbf X)}(t)
\right\}.
\]
For every \(t\in D\), the section
\[
B_t
:=
\{\omega\in\Omega:(\omega,t)\in B\}
\]
has probability zero. Since \([0,1]\setminus D\) has Lebesgue measure zero, Tonelli's theorem gives
\[
(\mathbb P\otimes\lambda)(B)=0.
\]
Consequently, there exists a measurable set \(A\subset\Omega\) with
\(\mathbb P(A)=1\) such that, for every \(\omega\in A\),
\[
\hat{\phi}_{(Y,\mathbf X)}(t,\omega)
\to
\phi_{(Y,\mathbf X)}(t)
\]
for Lebesgue-a.e. \(t\in[0,1]\).

Moreover,
\[
0\le
\bigl|
\hat{\phi}_{(Y,\mathbf X)}(t,\omega)
-
\phi_{(Y,\mathbf X)}(t)
\bigr|
\le 1
\]
for all \(t\in[0,1]\) and all \(\omega\in A\). Therefore, by the dominated convergence theorem,
for every \(\omega\in A\),
\[
\int_0^1
\bigl|
\hat{\phi}_{(Y,\mathbf X)}(t,\omega)
-
\phi_{(Y,\mathbf X)}(t)
\bigr|
\,dt
\to 0.
\]
Since \(\PP(A)=1\), this proves
\[
\int_0^1
\bigl|
\hat{\phi}_{(Y,\mathbf X)}(t)
-
\phi_{(Y,\mathbf X)}(t)
\bigr|
\,dt
\xrightarrow{a.s.}0.
\]
\end{proof}
As a plug-in estimator for $\kappa_{(Y,\mathbf{X})}(t)$, we define

\begin{align}
\hat{\kappa}_{(Y,\mathbf{X})}(t) 
&:= 1 - 3\int_0^t \bigl(1-\hat{\phi}_{(Y,\mathbf{X})}(s)\bigr)\,\mathrm{d}s.
\end{align}

Based on this construction, consistency of $\hat{\kappa}_{(Y,\mathbf X)}(t)$ follows.

\begin{theorem}\label{thm:kappa}
Let $(Y,\mathbf X)$ be a $(d+1)$-dimensional random vector with copula $C$, and let
$(Y_1,\mathbf X_1),\dots,(Y_n\mathbf ,X_n)$ be i.i.d.\ copies. Then
\[
\sup_{t\in[0,1]}
\bigl|\hat{\kappa}_{(Y,\mathbf X)}(t)-\kappa_{(Y,\mathbf X)}(t)\bigr|
\xrightarrow{a.s.} 0.
\]
Moreover,
\[
\hat{\kappa}_{(Y,\mathbf X)}(t)
=
1-3t+\frac{3}{n}\sum_{i=1}^n \left(t-\left|\frac{R_i-R_{N(i)}}{n+1}\right|\right)_+ .
\]
\end{theorem}

\begin{proof}
We write
\[
\hat{\kappa}_{(Y,\mathbf X)}(t)
=
1-3t+3\int_0^t \hat{\phi}_{(Y,\mathbf X)}(s)\,ds,
\qquad
\kappa_{(Y,\mathbf X)}(t)
=
1-3t+3\int_0^t \phi_{(Y,\mathbf X)}(s)\,ds.
\]
Hence,
\[
\sup_{t\in[0,1]}
|\hat{\kappa}_{(Y,\mathbf X)}(t)-\kappa_{(Y,\mathbf X)}(t)|
\le
3\int_0^1 |\hat{\phi}_{(Y,\mathbf X)}(s)-\phi_{(Y,\mathbf X)}(s)|\,ds.
\]
The claim follows from Theorem~\ref{thm:phi_L1}.

For the explicit representation, we start from
\[
\hat{\phi}_{(Y,\mathbf X)}(s)
=
\frac1n\sum_{i=1}^n \mathbf 1_{\left\{\left|\frac{R_i-R_{N(i)}}{n+1}\right|\le s\right\}}.
\]
Hence,
\begin{align*}
\int_0^t \hat{\phi}_{(Y,\mathbf X)}(s)\,ds
&=
\frac1n\sum_{i=1}^n \int_0^t \mathbf 1_{\left\{\left|\frac{R_i-R_{N(i)}}{n+1}\right|\le s\right\}}\,ds.
\end{align*}
For each fixed $i$, the integrand is zero for $s<\left|\frac{R_i-R_{N(i)}}{n+1}\right|$ and one otherwise, hence
\[
\int_0^t \mathbf 1_{\left\{\left|\frac{R_i-R_{N(i)}}{n+1}\right|\le s\right\}}\,ds
=
\left(t-\left|\frac{R_i-R_{N(i)}}{n+1}\right|\right)_+.
\]
This yields the claimed representation.
\end{proof}

\begin{cor}
Under the assumptions of Theorem~\ref{thm:kappa}, for every $t\in[0,1]$,
\[
\hat{\kappa}_{(Y,\mathbf X)}(t) \xrightarrow{a.s.} \kappa_{(Y,\mathbf X)}(t).
\]
\end{cor}

The final point that requires separate discussion is the boundary at $t=0$, since this value is of particular interest. However, due to the construction of the estimator, nearest-neighbor rank differences are always strictly positive and in fact bounded below by $(n+1)^{-1}$, the minimal spacing of normalized ranks.

\begin{proposition}\label{prop:bn}
There exists a deterministic sequence $(b_n)$ such that
\[
b_n\downarrow 0,
\qquad
b_n\ge \frac{1}{n+1},
\]
and
\[
\hat{\phi}_{(Y,\mathbf X)}(b_n)\xrightarrow{a.s.}\phi_{(Y,\mathbf X)}(0).
\]
\end{proposition}

\begin{proof}
Since $\phi_{(Y,\mathbf X)}$ is a distribution function, it is right-continuous and has at most countably many discontinuities. Hence, there exists a sequence of continuity points $(q_m)_{m\in\mathbb N}$ such that
\[
q_m\downarrow 0
\qquad\text{and}\qquad
\phi_{(Y,\mathbf X)}(q_m)\to \phi_{(Y,\mathbf X)}(0).
\]

By Theorem~\ref{thm:phi_pointwise}, for each fixed $m$,
\[
\hat{\phi}_{(Y,\mathbf X)}(q_m)\xrightarrow{a.s.}\phi_{(Y,\mathbf X)}(q_m).
\]
Hence,
\[
\sup_{n\ge N} \left|\hat{\phi}_{(Y,\mathbf X)}(q_m)-\phi_{(Y,\mathbf X)}(q_m)\right|
\xrightarrow{a.s.} 0 \quad (N\to\infty).
\]
In particular, we can choose a deterministic increasing sequence $(N_m)$ such that
\[
q_m \ge \frac{1}{N_m+1},
\qquad
\PP\!\left(
\sup_{n\ge N_m}
\left|\hat{\phi}_{(Y,\mathbf X)}(q_m)-\phi_{(Y,\mathbf X)}(q_m)\right|>\frac{1}{m}
\right)
\le 2^{-m}.
\]

Define $(b_n)$ by
\[
b_n := q_m
\qquad\text{whenever } N_m \le n < N_{m+1}.
\]
Then $b_n\downarrow 0$ and $b_n\ge \frac{1}{n+1}$.

By the Borel--Cantelli lemma, almost surely only finitely many of the events
\[
\sup_{n\ge N_m}
\left|\hat{\phi}_{(Y,\mathbf X)}(q_m)-\phi_{(Y,\mathbf X)}(q_m)\right|>\frac{1}{m}
\]
occur. Hence, almost surely, for all sufficiently large $m$ and all $n\ge N_m$,
\[
\left|\hat{\phi}_{(Y,\mathbf X)}(q_m)-\phi_{(Y,\mathbf X)}(q_m)\right|\le \frac{1}{m}.
\]

For $n$ with $N_m \le n < N_{m+1}$, we obtain
\[
\left|\hat{\phi}_{(Y,\mathbf X)}(b_n)-\phi_{(Y,\mathbf X)}(0)\right|
\le
\frac{1}{m}
+
\left|\phi_{(Y,\mathbf X)}(q_m)-\phi_{(Y,\mathbf X)}(0)\right|.
\]
Both terms converge to zero, which proves the claim.
\end{proof}

\begin{remark}
In practice, one may use $b_n = 1/n$ as a natural near-zero threshold. However, the proposition only guarantees the existence of at least one deterministic sequence $b_n\downarrow 0$ with the stated properties.
\end{remark}

\section{Data Examples and Simulations}
This section applies the previously introduced copula-based dependence measures in practice. 
The focus is on the function $\phi_{(Y,\mathbf{X})}(t)$, particularly at $t=0$, where positive values indicate a singular component and thus potential functional relationships between variables.

A wine dataset illustrates this: observed patterns suggest partial determinism between variables, reflected in a small but positive estimate of $\phi(0)$. A simulation study further shows that the estimators for $\phi_{(Y,X)}$ and its integral $\kappa$ are consistent, with $\kappa$ exhibiting lower median error but higher variability due to its smoothing nature.

\subsection{Data Examples}

As discussed in the theoretical part, the behavior of
\(\phi_{(Y,\XX)}\) near zero can also be used as a diagnostic tool for
empirical data. For variables that are naturally interpreted as continuous
measurements, one would typically not expect a pronounced concentration of
the Markov-product representation near the diagonal at very small thresholds.
A clearly positive near-zero value of \(\phi_{(Y,\XX)}\) may therefore point
to structural features of the data, such as preprocessing effects,
duplicates, or synthetic augmentation.

\begin{example}[Wine-quality data]

We analyze the current publicly available Kaggle balanced-classification
version of the Wine Quality data \cite{KaggleWine}, derived from the original
UCI/Cortez Wine Quality data \cite{UCIWineQuality}. The Kaggle data set used
here contains 21,000 observations and combines original observations with
synthetically generated ones. The original UCI/Cortez data consist of 1,599
red-wine and 4,898 white-wine observations. In the present example, we consider
\[
\XX=\texttt{total\_sulfur\_dioxide},
\qquad
Y=\texttt{sulphates}.
\]
Both variables are physicochemical measurements and are treated as continuous
quantities for the purpose of the present analysis.

In the ordinary scatterplot, the relationship between these variables does
not suggest an obvious deterministic or nearly deterministic structure. The
point cloud shows clustering and heterogeneity, but no visible functional
relation. In contrast, the corresponding Markov-product representation
exhibits a clear near-diagonal concentration; see
Figure~\ref{fig:wine_markov_product}.

\begin{figure}[ht]
\centering
\includegraphics[width=0.8\textwidth]{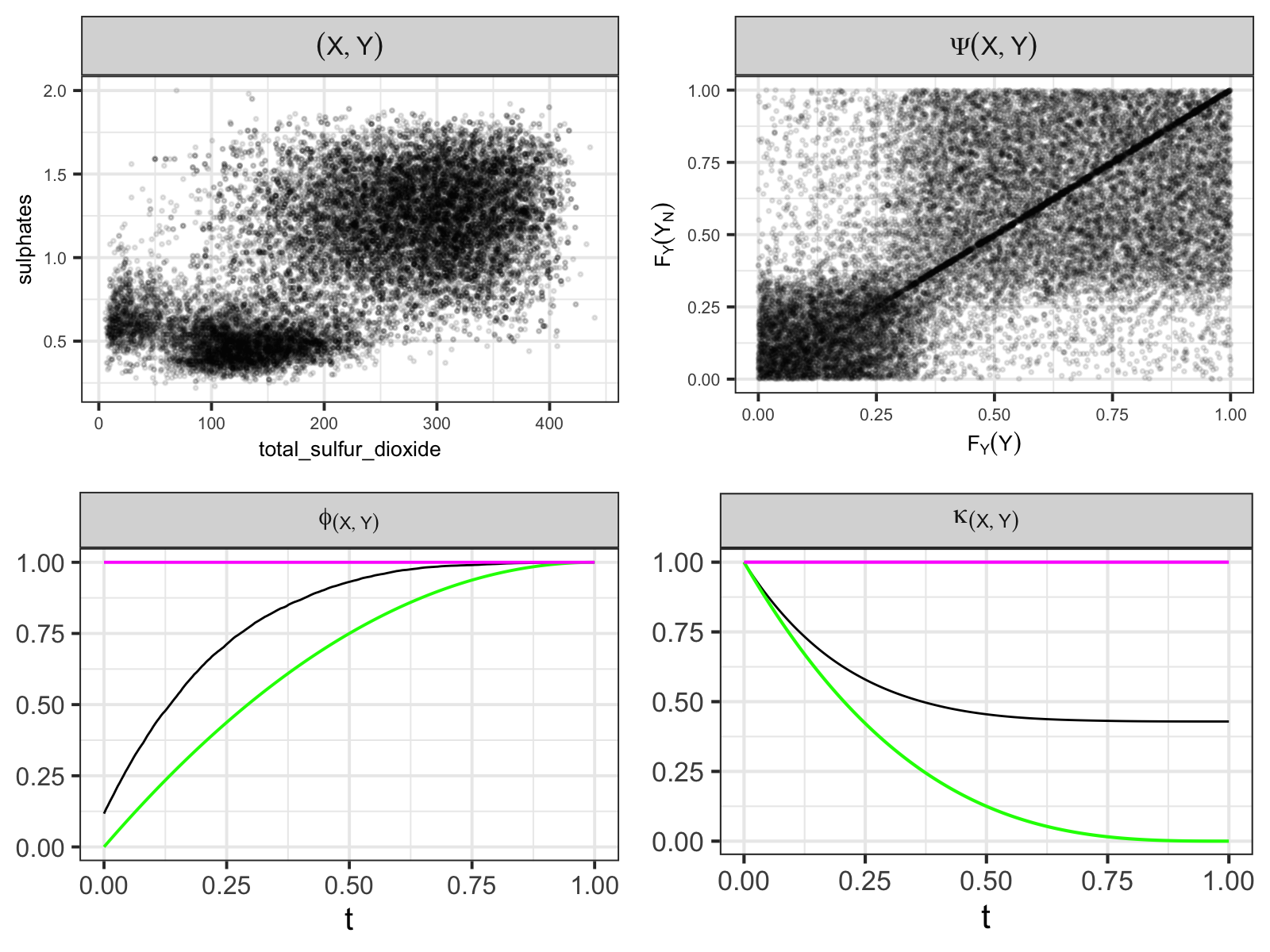}
\caption{Wine-quality data for \texttt{total\_sulfur\_dioxide} and
\texttt{sulphates}. The upper-left panel shows the original scatterplot,
while the upper-right panel shows the corresponding Markov-product
representation. The lower panels display the empirical
\(\phi_{(Y,\XX)}\)- and \(\kappa_{(Y,\XX)}\)-functions.}
\label{fig:wine_markov_product}
\end{figure}

This concentration is detected by \(\hat{\phi}_{(Y,\XX)}\)
near zero. Since the rank-based finite-sample estimator uses nearest-neighbor
indices with \(N(i)\neq i\), we report the value at the small threshold
\(b_n = 1/n\). In the present example,
\[
\hat{\kappa}_{(Y,\XX)}(1)\approx 0.43,
\qquad
\hat{\phi}_{(Y,\XX)}(b_n)\approx 0.09.
\]
Thus, \(\hat{\kappa}_{(Y,\XX)}(1)\) indicates a moderate level of directed
dependence, while \(\hat{\phi}_{(Y,\XX)}(b_n)\) quantifies the visible
near-diagonal concentration in the empirical Markov-product representation.

The source of this effect can be further assessed by comparing the Kaggle data
with the original UCI/Cortez data. The Kaggle file used here is a derived data
set and, according to its description, combines original observations with
synthetically generated ones to obtain a balanced classification data set.

A separate analysis of the original UCI/Cortez data supports this
interpretation. When the same variables are analyzed in the original data
alone, the corresponding near-zero value of the empirical
\(\phi_{(Y,\XX)}\)-function is close to zero, as expected for the selected
continuous measurements. The positive near-zero value in the Kaggle version
therefore suggests that the effect is related to structural features of the
derived dataset, such as synthetic augmentation, preprocessing, or duplication
effects, rather than to the original wine measurements alone.

Thus, the example illustrates how \(\phi_{(Y,\XX)}\) can reveal structural
features of empirical data that are not directly visible in the ordinary
coordinate representation.

\end{example}

\subsection*{Simulation study}

We now analyze the convergence behavior of the dependence function
\(\phi_{(Y,\XX)}\) and the associated area measure
\(\kappa_{(Y,\XX)}\). We simulate data from a bivariate Gaussian copula with
correlation parameter \(\rho=0.5\), that is, \(d=1\); see
Example~\ref{F_Yaus} for reference.

\textbf{Setup:} Let
\[
G:=\{0,0.01,\ldots,1\}
\]
be a uniform grid with 101 points in \([0,1]\). For each sample size \(n\), we
compute \(t\mapsto \hat{\phi}_{(Y,\XX)}(t)\) and
\(t\mapsto \hat{\kappa}_{(Y,\XX)}(t)\) based on 500 independent repetitions.
These are compared to the reference curves \(\phi_{(Y,\XX)}(t)\) and
\(\kappa_{(Y,\XX)}(t)\).

\textbf{Metric:} We measure the grid-based maximum absolute deviation by
\[
d_G(\phi_{(Y,\XX)},\hat{\phi}_{(Y,\XX)})
:=\max_{t\in G}
\left|\phi_{(Y,\XX)}(t)-\hat{\phi}_{(Y,\XX)}(t)\right|,
\]
and
\[
d_G(\kappa_{(Y,\XX)},\hat{\kappa}_{(Y,\XX)})
:=\max_{t\in G}
\left|\kappa_{(Y,\XX)}(t)-\hat{\kappa}_{(Y,\XX)}(t)\right|.
\]

\begin{figure}[h]
  \centering
  \includegraphics[width=0.85\textwidth]{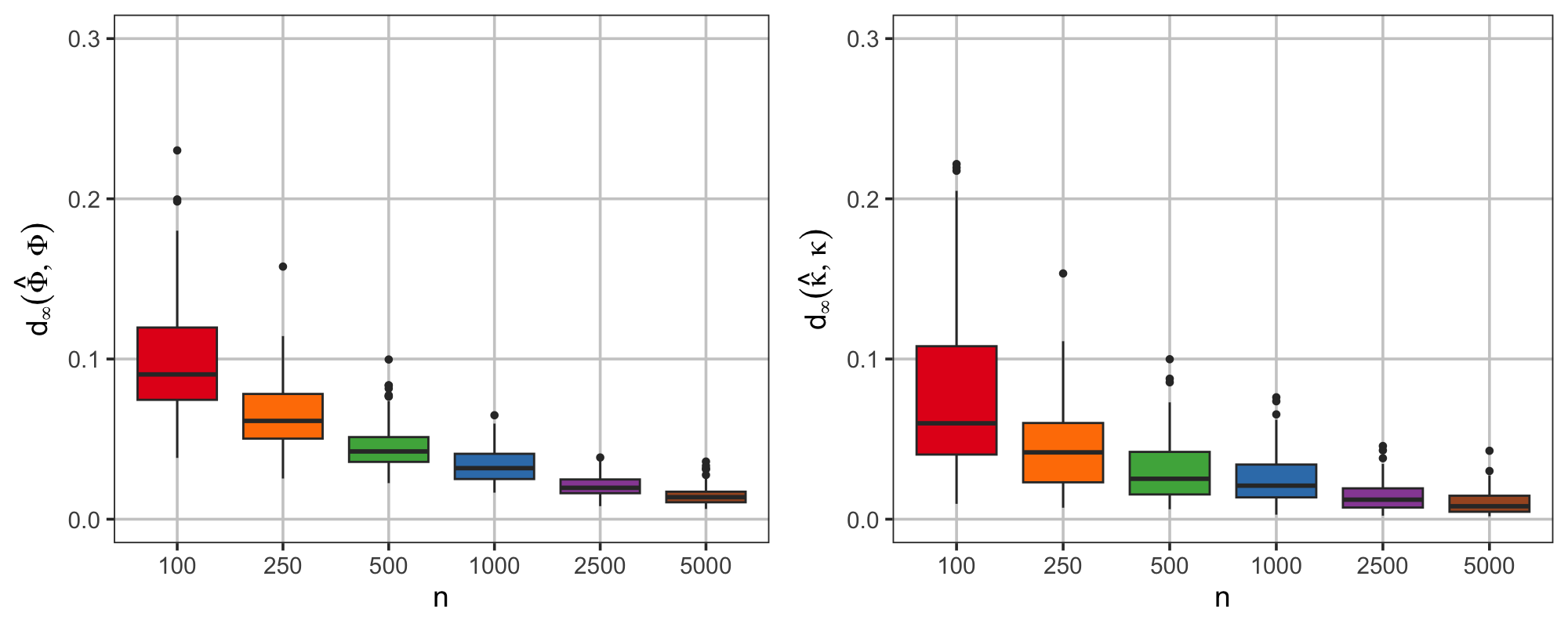}
  \caption{Boxplots of the grid-based maximum deviation
  \(d_G(\phi_{(Y,\XX)},\hat{\phi}_{(Y,\XX)})\) (left) and
  \(d_G(\kappa_{(Y,\XX)},\hat{\kappa}_{(Y,\XX)})\) (right) across different
  sample sizes \(n\). Each box is based on 500 independent repetitions.}
  \label{fig:convergence_behavior}
\end{figure}

In this continuous Gaussian setting, the grid-based maximum deviations decrease
empirically with increasing sample size. This is consistent with the theoretical
convergence results, while noting that the paper proves \(L^1\)-consistency for
\(\hat{\phi}_{(Y,\XX)}\) and uniform consistency for
\(\hat{\kappa}_{(Y,\XX)}\), rather than general uniform consistency for
\(\hat{\phi}_{(Y,\XX)}\).

Notably, the median deviation is systematically smaller for
\(\kappa_{(Y,\XX)}\) than for \(\phi_{(Y,\XX)}\), whereas the variance is larger
for the \(\kappa_{(Y,\XX)}\)-based measure. This can be attributed to the
integral structure of \(\hat{\kappa}_{(Y,\XX)}\): since
\(\hat{\kappa}_{(Y,\XX)}(t)\) is defined as an integral of
\(\hat{\phi}_{(Y,\XX)}\) from \(0\) to \(t\), local over- or underestimations of
\(\hat{\phi}_{(Y,\XX)}\) relative to \(\phi_{(Y,\XX)}(s)\) are partially
compensated as the integration progresses. This averaging effect---together
with the constant scaling factor---leads to smaller mean or median deviations.
However, if such compensation occurs only at larger values of \(t\), the
accumulated bias results in increased variability. This may explain why
\(\hat{\kappa}_{(Y,\XX)}\) exhibits lower median error but higher variance
compared to the pointwise estimator \(\hat{\phi}_{(Y,\XX)}\).

\section*{Conclusion}

In this paper, we developed a distribution-sensitive extension of Chatterjee’s correlation coefficient by analyzing the Markov product $(Y, Y')$ and introducing the dependence functions $\phi_{(Y,\XX)}(t)$ and $\kappa_{(Y,\XX)}(t)$. Rather than summarizing directed dependence by a single scalar, the proposed framework provides a functional description that captures both local and global structural features of the relationship between $Y$ and $\XX$. In particular, $\phi_{(Y,\XX)}(0)$ detects conditional atomic components in the Markov product. Such components may arise from functional, duplicated, discretized, preprocessed, or synthetically generated structure; they should therefore not be interpreted automatically as direct evidence of a deterministic physical relation. The function $\kappa_{(Y,\XX)}$ aggregates this information into a convex and well-behaved curve that generalizes Chatterjee’s $\xi$.

The theoretical analysis establishes sharp characterizations of independence and perfect dependence, clarifies the geometric interpretation of mass concentration along the diagonal after the Markov transformation, and demonstrates how different copula structures are reflected in the shape of $\phi_{(Y,\XX)}$ and $\kappa_{(Y,\XX)}$. This shows that dependence is not only a matter of strength but also of structural form.

From a statistical perspective, we proposed plug-in estimators based on nearest-neighbor ranks and proved strong consistency via weak convergence of empirical measures. Simulation results confirm the theoretical convergence behavior and illustrate the distinct error profiles of the pointwise estimator $\hat{\phi}_{(Y,\XX)}$ and the integrated estimator $\hat{\kappa}_{(Y,\XX)}$.

Overall, the framework extends scalar dependence measures to a richer functional object, improving interpretability and enabling a more nuanced analysis of directed stochastic dependence.
\newpage
\section*{Funding information}
This research was funded in whole by the Austrian Science Fund (FWF) [10.55776/P36155] project ReDim: Quantifying Dependence via Dimension Reduction. The author further acknowledges the support from the WISS 2025 project ‘IDA-lab Salzburg’ 20204-WISS/225/197–2019 and 20102-F1901166-KZP.

\section*{Author contributions}
The author confirms sole responsibility for the conception of the study,
the presented results, and the preparation of the manuscript.

\section*{Conflict of interest}
The author states no conflict of interest.

\section*{Ethical approval}
This research does not involve human participants, animals, or clinical data.

\section*{Data availability statement}

The data used in this study are publicly available. The empirical analysis
is based on a derived Kaggle version of the Wine Quality data available at \cite{KaggleWine}. This dataset is not identical to the original UCI/Cortez Wine
Quality data. The original Wine Quality data are available from the UCI
Machine Learning Repository and were introduced in Cortez et al.
\cite{UCIWineQuality}.

\section*{Declaration on the use of artificial intelligence}
ChatGPT was used for language editing, proofreading,
and formulation support during the preparation of the manuscript. The
mathematical results, conceptual development, data analysis, figures, and
scientific conclusions were produced and verified by the author. The author takes full responsibility for the content of the manuscript.



\begin{thebibliography}{99}

\bibitem{ansari2025continuity}
J.~Ansari and S.~Fuchs (2025).
\newblock On continuity of Chatterjee's rank correlation and related dependence measures.
\newblock \emph{Bernoulli}, to appear.
\newblock arXiv:2503.11390.
\newblock \url{https://arxiv.org/abs/2503.11390}.

\bibitem{ansari2026ordering}
J.~Ansari and S.~Fuchs (2025).
\newblock An ordering for the strength of functional dependence.
\newblock arXiv:2511.06498.
\newblock \url{https://arxiv.org/abs/2511.06498}.

\bibitem{deb2020chatterjee}
N.~Deb, P.~Ghosal, and B.~Sen (2020).
\newblock Measuring association on topological spaces using kernels and geometric graphs.
\newblock arXiv:2010.01768.
\newblock \url{https://arxiv.org/abs/2010.01768}.

\bibitem{tran2024rank}
L.~Tran and F.~Han (2024).
\newblock On a rank-based Azadkia--Chatterjee correlation coefficient.
\newblock arXiv:2412.02668.
\newblock \url{https://arxiv.org/abs/2412.02668}.

\bibitem{dalitz2024bias}
C.~Dalitz, J.~Arning, and S.~Goebbels (2024).
\newblock A simple bias reduction for Chatterjee's correlation.
\newblock \emph{J. Stat. Theory Pract.} \textbf{18}, Article 51.
\newblock DOI: \url{https://doi.org/10.1007/s42519-024-00399-y}.

\bibitem{lin2023boosting}
Z.~Lin and F.~Han (2023).
\newblock On boosting the power of Chatterjee's rank correlation.
\newblock \emph{Biometrika} \textbf{110}(2), 283--299.
\newblock DOI: \url{https://doi.org/10.1093/biomet/asac048}.

\bibitem{shi2021azadkia}
H.~Shi, M.~Drton, and F.~Han (2024).
\newblock On Azadkia--Chatterjee's conditional dependence coefficient.
\newblock \emph{Bernoulli} \textbf{30}(2), 851--877.
\newblock DOI: \url{https://doi.org/10.3150/22-BEJ1529}.

\bibitem{ansari2022extension}
J.~Ansari and S.~Fuchs (2022).
\newblock A direct extension of Azadkia \& Chatterjee's rank correlation to multi-response vectors.
\newblock arXiv:2212.01621.
\newblock \url{https://arxiv.org/abs/2212.01621}.

\bibitem{huang2023multivariate}
W.~Huang, Z.~Li, and Y.~Wang (2025).
\newblock A multivariate extension of Azadkia--Chatterjee's rank coefficient.
\newblock arXiv:2512.07443.
\newblock \url{https://arxiv.org/abs/2512.07443}.

\bibitem{azadkia2021simple}
M.~Azadkia and S.~Chatterjee (2021).
\newblock A simple measure of conditional dependence.
\newblock \emph{Ann. Stat.} \textbf{49}(6), 3070--3102.
\newblock DOI: \url{https://doi.org/10.1214/21-AOS2073}.

\bibitem{fuchs2026dimension}
S.~Fuchs and C.~Limbach (2026).
\newblock A dimension reduction for extreme types of directed dependence.
\newblock \emph{Dependence Modeling} \textbf{14}(1), Article 20250016.
\newblock DOI: \url{https://doi.org/10.1515/demo-2025-0016}.

\bibitem{fuchs2024JMVA}
S.~Fuchs (2024).
\newblock Quantifying directed dependence via dimension reduction.
\newblock \emph{J. Multivar. Anal.} \textbf{201}, Article 105266.
\newblock DOI: \url{https://doi.org/10.1016/j.jmva.2023.105266}.

\bibitem{durante2015principles}
F.~Durante and C.~Sempi (2015).
\newblock \emph{Principles of Copula Theory}.
\newblock Boca Raton, FL: Chapman and Hall/CRC.

\bibitem[\protect\citeauthoryear{Nelsen}{2006}]{nelsen2006}
R.~B. Nelsen (2006).
\newblock \emph{An Introduction to Copulas}.
\newblock 2nd ed., Springer.

\bibitem{fuchs2026semilinear}
S.~Fuchs, C.~Limbach, and F.~Sch\"urrer (2026).
\newblock On exact regions between measures of concordance and Chatterjee's rank correlation for lower semilinear copulas.
\newblock \emph{Int. J. Approx. Reason.} \textbf{189}, Article 109588.
\newblock DOI: \url{https://doi.org/10.1016/j.ijar.2025.109588}.

\bibitem{KaggleWine}
Taweilo (n.d.).
\newblock Wine Quality dataset -- Classification.
\newblock Kaggle dataset.
\newblock \url{https://www.kaggle.com/datasets/taweilo/wine-quality-dataset-balanced-classification}.

\bibitem{UCIWineQuality}
P.~Cortez, A.~Cerdeira, F.~Almeida, T.~Matos, and J.~Reis (2009).
\newblock Modeling wine preferences by data mining from physicochemical properties.
\newblock \emph{Decision Support Systems} \textbf{47}(4), 547--553.
\newblock DOI: \url{https://doi.org/10.1016/j.dss.2009.05.016}.

\end{thebibliography}

\newpage
\appendix
\section{Proofs and additional results}\label{sec:appendix}

\subsection*{Proof of Theorem 1.1}
\begin{proof}

By results in the literature \cite{ansari2025continuity}, we have
\[
\xi(Y,\mathbf X)
=
6\int_{\R}\PP(Y\ge y,Y'\ge y)\,d\PP^Y(y)-2.
\]

Using the decomposition
\[
\PP(Y\ge y,Y'\ge y)
=
1-\PP(Y<y)-\PP(Y'<y)+\PP(Y<y,Y'<y),
\]
and the fact that \(Y\) and \(Y'\) have the same continuous marginal distribution, we obtain
\begin{align*}
\int_{\R}\PP(Y\ge y,Y'\ge y)\,d\PP^Y(y)
&=
\int_{\R}
\Bigl(
1-2\PP(Y<y)+\PP(Y<y,Y'<y)
\Bigr)\,d\PP^Y(y) \\
&=
1-2\int_{\R}\PP(Y<y)\,d\PP^Y(y)
+
\int_{\R}\PP(Y<y,Y'<y)\,d\PP^Y(y) \\
&=
\int_{\R}\PP(Y<y,Y'<y)\,d\PP^Y(y),
\end{align*}
because, by continuity of \(Y\),
\[
\int_{\R}\PP(Y<y)\,d\PP^Y(y)
=
\frac12.
\]

Now let \(V=F_Y(Y)\) and \(V'=F_Y(Y')\). By continuity of \(F_Y\), both \(V\) and \(V'\) are uniformly distributed on \((0,1)\). Using the change of variables \(t=F_Y(y)\), we get
\[
\int_{\R}\PP(Y<y,Y'<y)\,d\PP^Y(y)
=
\int_0^1 \PP(V<t,V'<t)\,dt.
\]
By Fubini's theorem,
\[
\int_0^1 \PP(V<t,V'<t)\,dt
=
\mathbb E\left[\int_0^1 \mathds{1}_{\{V<t,V'<t\}}\,dt\right]
=
\mathbb E\left[\int_0^1 \mathds{1}_{\{t>\max(V,V')\}}\,dt\right]
=
\mathbb E[1-\max(V,V')].
\]
Since $V$ and $V'$ are uniformly distributed,
\[
\mathbb E[1-\max(V,V')]
=
\mathbb E[\min(V,V')],
\]
because
\[
\min(V,V')+\max(V,V')=V+V'
\quad\text{and}\quad
\mathbb E[V+V']=1.
\]
Thus
\[
\int_{\R}\PP(Y\ge y,Y'\ge y)\,d\PP^Y(y)
=
\mathbb E[\min(V,V')].
\]

Since
\[
\min(V,V')
=
V-(V-V')_+,
\]
we obtain
\[
\mathbb E[\min(V,V')]
=
\mathbb E[V]-\mathbb E[(V-V')_+]
=
\frac12-\mathbb E[(V-V')_+].
\]
If \((Y,Y')\) is exchangeable, then \((V,V')\) is exchangeable as well, and hence
\[
\mathbb E[(V-V')_+]
=
\mathbb E[(V'-V)_+]
=
\frac12\mathbb E|V-V'|.
\]
Therefore,
\[
\int_{\R}\PP(Y\ge y,Y'\ge y)\,d\PP^Y(y)
=
\frac12-\frac12\mathbb E|V-V'|.
\]

Plugging this into the representation of \(\xi\) yields
\[
\xi(Y,\mathbf X)
=
3\Bigl(\mathbb E(1-|V-V'|)\Bigr)-2.
\]

Finally, since \(0\le |V-V'|\le 1\), Fubini's theorem gives
\[
\int_0^1 \PP(|V-V'|\le s)\,ds
=
\mathbb E[1-|V-V'|].
\]
Hence,
\[
\xi(Y,\mathbf X)
=
3\int_0^1 \PP\bigl(|F_Y(Y)-F_Y(Y')|\le s\bigr)\,ds-2.
\]

\end{proof}
\end{document}